\documentclass{article}

\def\cal{\mathcal}
 \catcode`@=11
\usepackage{amsthm,amssymb,amsfonts,mathrsfs,color}
\usepackage{amsmath}
\catcode`@=11 \@addtoreset{equation}{section}

\catcode`@=12

\newcommand{\RN}[1]{%
	\textup{\uppercase\expandafter{\romannumeral#1}}%
}

\newtheorem{Theorem}{Theorem}[section]
\newtheorem{Proposition}{Proposition}[section]
\newtheorem{Lemma}{Lemma}[section]
\newtheorem{Corollary}{Corollary}[section]
\theoremstyle{definition}

\newtheorem{Definition}{Definition}[section]
\newtheorem{Remark}{Remark}

\def\R{{\mathbb{R}}}

\def\cU{{\mathcal{U}}}

\def\t\cU{{\widetilde{{\mathcal{U}}}}}

\def\Ho{{H^2_{a,0}(0,1)}}
\def\H2{{H^2_{a}(0,1)}}
\newcommand\norm[1]{\left\lVert#1\right\rVert}

\def\ds{\displaystyle}

\title {New results on controllability and stability for degenerate Euler-Bernoulli type equations}
\author{{{\sc Alessandro Camasta}\thanks{The author is a member of the  {\it Gruppo Nazionale per l'Analisi Ma\-te\-matica, la Probabilit\`a e le loro Applicazioni (GNAMPA)} of the Istituto Nazionale di Alta Matematica (INdAM), a member of {\it UMI ``Modellistica Socio-Epidemiologica (MSE)''} and he is partially supported  by  the GNAMPA project 2023 {\it Modelli differenziali per l'evoluzione del clima e i suoi impatti}. He is also a member of the project {\it Mathematical models for interacting dynamics on networks (MAT-DYN-NET)
				CA18232.}}}\\
	Department of Mathematics\\ University of Bari Aldo Moro\\
	Via
	E. Orabona 4\\ 70125 Bari - Italy\\ e-mail: alessandro.camasta@uniba.it\\
	{\sc Genni Fragnelli}\thanks{The author is a member of the  {\it Gruppo Nazionale per l'Analisi Ma\-te\-matica, la Probabilit\`a e le loro Applicazioni (GNAMPA)} of the Istituto Nazionale di Alta Matematica (INdAM), a member of {\it UMI ``Modellistica Socio-Epidemiologica (MSE)''} and is supported by FFABR {\it Fondo per il finanziamento delle attivit\`a base di ricerca} 2017 and by  the GNAMPA project 2023 {\it Modelli differenziali per l'evoluzione del clima e i suoi impatti}.}\\
	Department of Ecology and Biology\\ Tuscia University\\
	Largo dell'Universit\`a, 01100 Viterbo - Italy\\ e-mail: genni.fragnelli@unitus.it}

\date{}

\begin{document}
	
	\maketitle

	\vspace{0.3cm}
	
	\centerline{ {\it  }}
	
	\begin{abstract}
	In this paper we study the controllability and the stability for a degenerate beam equation in divergence form via the energy method. The equation is  clamped at the left end
and controlled by applying a shearing
force  or a damping at the right end. 
	\end{abstract}
	
	
	\noindent 2000AMS Subject Classification: 35L10, 35L80, 93B05, 93B07, 93D23, 93D15

	\section{Introduction}
	Starting from a paper of G. Leugering \cite{leugering1990}, we consider the following viscoelastic beam of length $1$:
	\begin{equation}\label{E1}
	y_{tt}(t,x)+(a(x)y_{xx}(t,x))_{xx}=0, \quad (t,x) \in (0, +\infty) \times (0,1)
	\end{equation}
	which is clamped at the left end, $x=0$,
and controlled by applying a shearing
force  or a damping at the right end, $x = 1$. In particular, at $x=1$, we consider the following conditions
\begin{equation}\label{E2}
y(t,1)=0 \quad \text{and} \quad y_x(t,1)= f(t), 
\end{equation} where $f\in L^2_{loc}[0,+\infty)$, or
\begin{equation}\label{E3}\begin{cases}
	\beta y(t,1)-(ay_{xx})_x(t,1)+y_t(t,1)=0, \\
		\gamma y_x(t,1)+(ay_{xx})(t,1)+y_{tx}(t,1)=0, \\
	\end{cases}
	\end{equation}
	for all $t\in (0,T)$, $T>0$, where $\beta$ and $\gamma$ are non negative constants.
The novelty of this paper is that the function $a:[0,1]\to\mathbb{R}$ degenerates at the extreme point $x=0$; in particular, we consider two cases: the weakly degenerate case and the strongly degenerate one. The degeneracy of our problem is quantified by a suitable real parameter that belongs to the interval $(0,2)$ for technical reasons. More precisely, we have the following definitions:
	\begin{Definition}\label{Def1}
		A function $a$ is {\it weakly degenerate at $0$}, $(WD)$ for short, if $a\in\mathcal{C}[0,1]\cap\mathcal{C}^1(0,1]$, $a(0)=0$, $a>0$ on $(0,1]$ and  if
		\begin{equation}\label{sup}
			K:=\sup_{x\in (0,1]}\frac{x|a'(x)|}{a(x)},
		\end{equation}
		then $K\in (0,1)$.
	\end{Definition}
	\begin{Definition}\label{Def2}
		A function $a$ is {\it strongly degenerate at $0$}, $(SD)$ for short,  if $a\in\mathcal{C}^1[0,1]$, $a(0)=0$, $a>0$ on $(0,1]$ and in \eqref{sup} we have $K\in [1,2)$.
	\end{Definition} 
Due to the degeneracy at $x=0$, the natural boundary conditions at $x=0$ are
	 \begin{equation}\label{E4}
	 y(t,0)=0\quad  \text{ and } \quad \begin{cases}
			y_x(t,0)=0, &\text{ if } a \text{ is (WD)},\\
			(ay_{xx})(t,0)=0, &\text{ if } a \text{ is (SD)},\\ 
			\end{cases}
	 \end{equation}
for all $t \in (0, T),$ $T>0$. Hence, in the present paper we study two problems:
the first one is the controllability, i.e. we want to find  conditions such that there exists a control $f\in L^2(0,T)$, where $T$ is sufficiently large, so that the solution of \eqref{E1}, \eqref{E2}, \eqref{E4} with initial conditions 
\begin{equation}\label{IC}
y(0,x)=y_0(x)\quad \text{ and } \quad y_t(0,x)=y_1(x),
\end{equation}
$x \in (0,1)$, satisfies
		\[
			y(T,x)=y_t(T,x)=0\,\,\,\,\,\,\forall\; x\in (0,1).
		\]
		The second problem is the stability, i.e. we want to find conditions such that the solution of \eqref{E1}, \eqref{E3}, \eqref{E4} and \eqref{IC}
		satisfies
		\[
		\lim_{t \rightarrow +\infty}E_y(t) =0,
		\]
		where $E_y$ is the energy associated to the problem.

Problems similar to \eqref{E1} are considered in several papers and are studied from a lot of point of views, see, for example, \cite{shubov}. However, in a lot of cases the equation is in non divergence form (see, for example, \cite{behn}, \cite{biselli}, \cite{chen1}, \cite{chen2}, \cite{chen}, \cite{coleman}, \cite{sandilo}).
In particular, in \cite{chen1} and \cite{chen} the following Euler-Bernoulli beam equation is considered
\[
my_{tt}(t,x) + EIy_{xxxx}(t,x)=0, \quad x \in (0,1), \;t>0,
\]
with clamped conditions  \eqref{E2} at the left end $x=0$ 
and with dissipative conditions at the right end
\begin{equation}\label{shear}
\begin{cases}
 -EIy_{xxx}(t,1)+\mu_1y_t(t,1)=0, & \mu_1 \ge0,\\
EIy_{xx}(t,1)+\mu_2y_{tx}(t,1)=0, & \mu_2 \ge0.
\end{cases}
\end{equation}
Here  $m$ is the mass density per unit length and
$EI$ is the flexural rigidity coefficient. Moreover,  the following variables have engineering meanings:
$y$ is the vertical displacement, $y_t$ is the velocity, $y_x$ is the rotation, $y_{tx}$ is the angular velocity, $-EIy_{xx}$ is the bending moment and $-EIy_{xxx}$ is the shear.  In particular, the boundary conditions \eqref{shear}  mean that the shear is proportional to the velocity and the bending moment is negatively proportional to the angular moment.

However, in all the previous papers the equation is {\it non degenerate}. The first results on controllability and stabilization for a {\it degenerate} beam equation but in {\it non divergence form} can be found in \cite{CF_Beam} and \cite{CF_Beam_Stability}, respectively.  Similar results for degenerate {\it wave equations}  can be found in \cite{ACL} (see also the arxiv version published in 2015), for the divergence case, and \cite{BFM} and \cite{Stability_Genni_Dimitri}, for the non divergence one.
To our knowledge, this is the first paper where the equation {\it degenerates} in a point of the space domain and it is in {\it divergence form}. Clearly,  controllability and stability results for a problem in divergence form can also be obtained from the ones of the problem in non divergence form, simply rewriting
\begin{equation}\label{divnondiv}
(a(x)y_{xx}(t,x))_{xx}=a''(x)y_{xx}(t,x)+ 2a'(x)y_{xxx}(t,x)+ a(x)y_{xxxx}(t,x).
\end{equation}
Thus, if we study well posedness and stability for \eqref{E1} rewritten using \eqref{divnondiv}, we have to require additional regularity on the function $a$ and on the solution $y$. Moreover,  the reference space for  \eqref{E1} is the natural $L^2(0,1)$; on the contrary, the reference space for the equation in
non divergence form \eqref{divnondiv}  is a suitably weighted $L^2$ - space. Thus, one can understand the importance of this paper: we prove controllability and stability results under basic assumptions on the degenerate function $a$ in the natural $L^2$ - space.

From the methodological point of view the strategy employed to analyze \eqref{E1} is based on an energy approach and direct methods such as multipliers. For the first problem we consider the energy of the associated homogeneous adjoint problem, for the second one we consider the energy of \eqref{E1}.  We underline again that in both cases we assume that the constant $K$ appearing in Definitions \ref{Def1} and \ref{Def2} is strictly less than $2$ for technical reasons. The case $K \ge 2$ is an open problem.

The paper is organized as follows. In Section \ref{section 2} we introduce notations, hypotheses and preliminary results used throughout the paper. In Section \ref{section 3} we consider \eqref{E1}, \eqref{E2}, \eqref{E4} and \eqref{IC} and  we prove that the energy associated to the solution of the adjoint problem is constant through time; then we establish some estimates from above and below for the energy and, thanks to these estimate and to the boundary observability, we prove that the original problem has a unique solution
by transposition, which is null controllable. In Section \ref{section 4}  we consider  \eqref{E1}, \eqref{E3}, \eqref{E4} and \eqref{IC} and we study the boundary stabilization problem proving its well posedness together with its exponential stability. However, the case $\gamma = 0$ and/or $\beta=0$ in the strongly degenerate case is still an open problem. Section \ref{appendice} is devoted to the proof of the generation theorems that are crucial to derive the well posedness of the considered problems.
	\section{Preliminary results}\label{section 2}
	In this section we will introduce some preliminary results that will be crucial for the following. As a first step, we consider the next spaces introduced in \cite{CF_Wentzell}.
	
	If the function \underline{$a$ is (WD)}, we consider:
	\[
		\begin{aligned}
			V^2_a(0,1):=\{u\in H^1(0,1):& \; u' \text{ is absolutely continuous in [0,1]},\\
			& \sqrt{a}u''\in L^2(0,1)\};
		\end{aligned}
	\]
on the other hand, if \underline{$a$ is (SD)}, then 
	\[
	\begin{aligned}
			V^2_a(0,1):=\{u\in H^1(0,1):& \; u' \text{ is locally absolutely continuous in } (0,1],\\
			& \sqrt{a}u''\in L^2(0,1)\}.
	\end{aligned}\]
	In both cases we consider on $V^2_a(0,1)$ the norm
	\begin{equation}\label{normadiv}
		\|u\|^2_{2,a}:= \|u\|^2_{L^2(0,1)}+ \|u'\|^2_{L^2(0,1)}+ \|\sqrt{a}u''\|^2_{L^2(0,1)} \quad \forall \; u \in V^2_a(0,1),
	\end{equation}
which is equivalent  to the following one
		\begin{equation}\label{normadiv1}
		\|u \|_{2}^2:= \|u\|^2_{L^2(0,1)}+ \|\sqrt{a}u''\|^2_{L^2(0,1)} \quad \forall \; u \in V^2_a(0,1)
	\end{equation}
(see \cite[Propositions 2.1 and 2.2]{CF_Wentzell}). 
	Using the previous spaces, it is possible to prove the following Gauss-Green formula: setting
\[
\mathcal Q(0,1):=\{u\in V^2_a(0,1): au''\in H^2(0,1)\},
\]
one has
\begin{equation}\label{GG0}
\int_0^1 (au'')''v\,dx =[(au'')'v]_{x=0}^{x=1}  - [au''v']_{x=0}^{x=1} + \int_0^1 au''v''dx
\end{equation}
for all   $(u,v) \in \mathcal Q(0,1) \times V^2_a(0,1)$ 
(see \cite[Lemmas 2.1 and 2.3]{CF_Neumann}). 

In $\mathcal Q(0,1)$ the next lemmas hold.
\begin{Lemma}\label{boundaryconditions_WD}
Assume $a$ (WD) and $y \in \mathcal Q(0,1)$. The following assertions are true:
\begin{enumerate}
\item If $y'(0)=0$ or $(ay'')'(0)=0$, then
$
		\lim_{\delta \rightarrow 0^+}y'(\delta)(ay'')'(\delta)= 0$.
		\item  If $y'(0)=0$ or $(ay'')(0)=0$, then 
		$\lim_{\delta \rightarrow 0^+} y'(\delta)(ay'')(\delta)=0$.
		\item $\lim_{\delta \rightarrow 0^+}a(\delta)(y')^2(\delta)=0.$
		\item $\lim_{\delta\to 0^+}\delta a(\delta)(y'')^2(\delta)=0$.
		\end{enumerate}
\end{Lemma}
\begin{proof}  Since $a$ is (WD) and $y \in \mathcal Q(0,1)$, one has $ay'' \in \mathcal C^1[0,1]$ and $y'$ absolutely continuous in $[0,1]$. Thus, one has immediately
	\[
		\lim_{\delta \rightarrow 0^+}y'(\delta)(ay'')'(\delta)=\lim_{\delta \rightarrow 0^+} y'(\delta)(ay'')(\delta)=\lim_{\delta \rightarrow 0^+}a(\delta)(y')^2(\delta)=0.
		\] 
		Thus it remains to prove the last point.  To this aim observe that (\ref{sup}) implies that  the function
	\begin{equation}\label{crescente}
		x\mapsto\frac{x^\gamma}{a(x)}
	\end{equation}
is non decreasing in $(0,1]$ for all $\gamma\ge K$, in particular
		\begin{equation}\label{ipotesi limite}
		\lim_{x\to 0}\frac{x^\gamma}{a(x)}=0
	\end{equation}
	for all $\gamma >K$. Thus, using the fact that $ay'' \in \mathcal C^1[0,1]$ and  \eqref{ipotesi limite}, one has
		 	\begin{equation}\label{stimabt2}
			\lim_{\delta\to 0^+}\delta a(\delta)(y'')^2(\delta)=\lim_{\delta\to 0^+}\frac{\delta}{a(\delta)}(ay'')^2(\delta)=0.
		\end{equation}

\end{proof}

\begin{Lemma}\label{boundaryconditions_SD}
Assume $a$ (SD) and $y\in \mathcal Q(0,1)$. The following assertions are true:
\begin{enumerate}
\item $
		\lim_{\delta \rightarrow 0^+}\delta  y'(\delta)=0.
	$
\item If $(ay'')(0)=0$, then
$
		\lim_{\delta \rightarrow 0^+} y'(\delta)(ay'')(\delta)=0.
$
		\item  If $(ay'')(0)=0$, then
		$\lim_{\delta \rightarrow 0^+}a(\delta)(y')^2(\delta)=0.$

\item 	If $(ay'')(0)=0$, then
$
	\lim_{\delta\to 0^+}\delta a(\delta)(y'')^2(\delta)=0.
$
\end{enumerate}
\end{Lemma}	
\begin{proof}
1. As a first step we will prove that $\exists \lim_{\delta \rightarrow 0^+} \delta y'(\delta)\in \R$. To this aim, we rewrite $\delta y'(\delta)$  as
		\[
		\delta y'( \delta) = y'(1) -\int_\delta^1 (x y')'(x)dx =y'(1) -\int_\delta^1  y'(x)dx - \int_\delta^1 x y''(x)dx.
		\]
		Since $y' \in L^2(0,1)$ and $| x y''|= \frac{x}{\sqrt a}|\sqrt {a}y''| \le \frac{1}{\sqrt{a(1)}}|\sqrt {a}y''| \in L^2(0,1)$, by the absolute continuity of the integral, one has that $\exists \lim_{\delta \rightarrow 0^+} \delta y'(\delta)=L\in \R$. Assume $L\neq0$, then there exists $C>0$ such that
		\[
		|y'(\delta)| \ge \frac{C}{\delta}
		\]
		for all $\delta$ in a neighbourhood of $0$. Since $\ds\frac{1}{\delta}\not\in L^1(0,1)$, one can conclude that $y' \not \in L^1(0,1)$, but this is not possible since $y' \in L^2(0,1)$. Hence $L=0$.
		
		2.	Clearly,
			\[
			y'(\delta)(ay'')(\delta)=y'(1)(ay'')(1) -\int_\delta^1 (y'ay'')'dx =y'(1)(ay'')(1) -\int_\delta^1 a(y'')^2dx -\int_\delta^1 y'(ay'')'dx.
			\]
			By the absolute continuity of the integral, using the fact that $a(y'')^2$ and  $y'(ay'')'$ belong to $L^1(0,1)$, one has that there exists $\lim_{\delta \rightarrow 0}y'(\delta)(ay'')(\delta)=L \in \R$. As before, if $L\neq 0$, then there exists $C>0$ such that
		\begin{equation}\label{stimatb}
		|y'(\delta)(ay'')(\delta)| \ge C
		\end{equation}
		for all $\delta$ in a neighbourhood of $0$.
		Moreover, 
			\begin{equation}\label{tondino}
		|(ay'')(\delta)|\le \int_0^\delta |(ay'')'(x)|dx \le \sqrt{\delta}\|(ay'')'\|_{L^2(0,1)}.
		\end{equation}
		Hence, by \eqref{stimatb} and \eqref{tondino}, we have
		\[
		C \le |y'(\delta)ay''(\delta)| \le |y'(\delta)|\sqrt{\delta}\|(ay'')'\|_{L^2(0,1)}.
		\]
		This implies
		\begin{equation}\label{tondino1}
		|y'(\delta)|\ge \frac{C}{\sqrt{\delta}\|(ay'')'\|_{L^2(0,1)}}
		\end{equation}
		for all  $\delta$ in a neighbourhood of $0$,
		in contrast to the fact that $y' \in L^2(0,1)$. Hence $L=0$.

				3. 
				To obtain the thesis we rewrite $a(\delta)(y')^2(\delta)$ in the following way
			\[
			\begin{aligned}
			a(\delta)(y')^2(\delta)&=a(1)(y')^2(1) - \int_\delta^1 (a(y')^2)'(s)ds = a(1)(y')^2(1) - \int_\delta^1 (a'(y')^2)(s)ds \\&
			-  2\int_\delta^1 (ay'y'')(s)ds.
			\end{aligned}
			\]
		Since $a'(y')^2$ and $ay'y''$ belong to $L^1(0,1)$, we have that $\exists \lim_{\delta \rightarrow 0^+} a(\delta)(y')^2(\delta)=L \in \R$. If $L\neq0$, then there exists $C>0$ such that
		\[
		|(y')^2(\delta)| \ge \frac{C}{a(\delta)}
		\]
		for all $\delta$ in a neighbourhood of $0$. Since $\ds\frac{1}{a}\not\in L^1(0,1)$ (see \cite{BoCaFr}), one can conclude that $y'\not \in L^2(0,1)$. Hence $L=0$.
	
	4.
				If $K=1$, then, using the fact that $\ds\frac{\delta}{a(\delta)} \le \frac{1}{a(1)}$ and proceeding as in \eqref{stimabt2}, one has immediately 
	the thesis since $ay'' \in \mathcal C^1[0,1]$ and $(ay'')(0)=0$.
	On the other hand, if $K \in (1,2)$, then, using the Taylor formula, we have
		\[
		a(\delta) y''(\delta)= (ay'')'(\xi_\delta) \delta + o(\delta),
		\]
		for $\xi_\delta \in (0, \delta)$;
		hence
		\begin{equation}\label{stimabt4}
		\begin{aligned}
			\delta a(\delta)(y'')^2(\delta) &= \delta y''(\delta)((ay'')'(\xi_\delta) \delta + o(\delta))\\
			& = (ay'')'(\xi_\delta) \delta^2 y''(\delta) + o(\delta^2)y''(\delta).
	\end{aligned}
		\end{equation}
		Now, 
$
		\lim_{\delta \rightarrow 0}	\delta ^2y''(\delta)= \lim_{\delta \rightarrow 0}	\frac{\delta ^2}{a(\delta)}(ay'')(\delta)=0
		$
and,		analogously, 
$
			\lim_{\delta \rightarrow 0}o(\delta^2)y''(\delta) =0;
$ moreover, using the fact that  $ay'' \in \mathcal C^1[0,1]$, one has  $	\lim_{\delta \rightarrow 0} \delta ^2y''(\delta)(ay'')'(\xi_\delta)= 0$. 
Thus, passing to the limit in \eqref{stimabt4}, we have the thesis.
			 
 (Observe that the previous proof holds also if $K=1$).	

\end{proof}

Finally, we recall the following Hardy-Poincaré inequality which will be essential in the next sections:
\begin{Proposition}[Proposition 2.1, \cite{BoCaFr}]\label{HP}
Assume that $a\, : \, [0,1] \rightarrow \mathbb{R}_+$ is in
${\cal C}[0,1]$, $a(0)=0$, $a>0$ on $(0,1]$.
If there exists
$\theta \in (0,1)$ such that the function
$$
x \mapsto \dfrac{a(x)}{x^{\theta}} \mbox { is
non increasing in } (0,1],
$$
\noindent then  for any
function $w$, locally absolutely continuous on $(0,1]$, continuous
at $0$ and satisfying

$$
w(0)=0 \,,\, \mbox{and } \int_0^1 a(x)|w^{\prime}(x)|^2 \,dx <
+\infty, \,
$$
\noindent the following inequality holds

\[
\int_0^1 \dfrac{a(x)}{x^2}w^2(x)\, dx \leq C_{HP}\, \int_0^1 a(x)
|w^{\prime}(x)|^2 \,dx, \,
\]
where $C_{HP}:= \ds\frac{4}{(1-\theta)^2}.$
\end{Proposition}

	\section{The boundary controllability problem}\label{section 3}
	The first problem considered in this paper is a boundary controllability problem. In particular, we will consider the following system
	\begin{equation}\label{(P)}
		\begin{cases}
			u_{tt}(t,x)+(au_{xx})_{xx}(t,x)=0, &(t,x)\in Q_T,\\
			u(t,0)=0,\,\,&t\in (0,T),\\
			\begin{cases}
			u_x(t,0)=0, &\text{ if } a \text{ is (WD)},\\
			(au_{xx})(t,0)=0, &\text{ if } a \text{ is (SD)},\\ 
			\end{cases}&t\in (0,T),\\
			u(t,1)=0,\,\,u_x(t,1)=f(t), &t\in (0,T),\\
			u(0,x)=u_0(x),\,\,u_t(0,x)=u_1(x),&x\in(0,1),
		\end{cases}
	\end{equation}
	where $T>0$, $Q_T:=(0,T) \times (0,1)$,  $f$ is a control function that acts on the non degeneracy boundary point and it is used to drive the solution to equilibrium at a given time $T>0$. In this framework we look for conditions under which, given initial data $(u_0,u_1)$ in a suitable space, there exists a control $f$ such that
	\begin{equation*}
		u(T,x)=u_t(T,x)=0,\,\,\,\,\,\,\,\quad\forall\; x\in (0,1).
	\end{equation*}
	\subsection{Well posedness for the problem with homogeneous boundary conditions}\label{subsect2}
	In this subsection we study the well posedness of the following degenerate hyperbolic problem  with Dirichlet boundary conditions
	\begin{equation}\label{(P_1)}
		\begin{cases}
			y_{tt}(t,x)+(ay_{xx})_{xx}(t,x)=0, &(t,x)\in (0,+\infty) \times (0,1),\\
			y(t,0)=y(t,1)=0,&t\in (0,+\infty),\\
			y_x(t,1)=0,&t\in (0,+\infty),\\
			\begin{cases}
			y_x(t,0)=0, &\text{ if } a \text{ is (WD)},\\
		(ay_{xx})(t,0)=0, &\text{ if } a \text{ is (SD)},\\ 
			\end{cases}, &t\in (0, + \infty),\\
			y(0,x)=y^0_T(x),&x\in(0,1),\\
			y_t(0,x)=y^1_T(x),&x\in(0,1)
		\end{cases}
	\end{equation}
	and we give some preliminary results needed for the following.
	
We underline that the choice of denoting initial data with $T$-dependence is connected to the approach for null controllability used in Subsection \ref{section 3}.3.
	
In order to study the well posedness of (\ref{(P_1)}), thanks to the spaces introduced in Section \ref{section 2}, we consider the following Hilbert spaces, which are different according to the kind of degeneracy of $a$. 

In particular, if the function \underline{$a$ is (WD)}, then we consider:
	\[		\begin{aligned}
			H^2_a(0,1):&=\{u \in V^2_a(0,1): u(0)=u(1)=0\}\\
			&=\{u\in H^1_0(0,1): u' \text{ is absolutely continuous in [0,1]},\\
			& \quad \;\; \sqrt{a}u''\in L^2(0,1)\};
		\end{aligned}
	\]
on the other hand, if \underline{$a$ is (SD)}, then 
	\[
		\begin{aligned}
			H^2_a(0,1):&=\{u \in V^2_a(0,1): u(0)=u(1)=0\}\\
			&=\{u\in H^1_0(0,1): u' \text{ is locally absolutely continuous in } (0,1],\\
			& \quad \; \;\sqrt{a}u''\in L^2(0,1)\}.
	\end{aligned}\]
	In both cases we consider on $H^2_a(0,1)$ the norms
	$
		\|\cdot\|_{2,a}$ or
		$
		\|\cdot\|_{2}$ defined in \eqref{normadiv} and in \eqref{normadiv1}, respectively.
In every case, setting
\[
\mathcal Z(0,1):=\{u\in \H2: au''\in H^2(0,1)\},
\]
the Gauss-Green formula \eqref{GG0} becomes
\begin{equation}\label{GG}
\int_0^1 (au'')''v\,dx = - [au''v']_{x=0}^{x=1} + \int_0^1 au''v''dx
\end{equation}
for all   $(u,v) \in \mathcal Z(0,1) \times \H2$.
Now, define

\[
		\begin{aligned}
			H^2_{a,0}(0,1):&=\{u\in H^2_a(0,1): u'(0)=u'(1)=0\}\\
			&=\{u\in H^1_0(0,1): u' \text{ is absolutely continuous in [0,1]},\\
			& \quad \quad u'(0)=u'(1)=0, \sqrt{a}u''\in L^2(0,1)\}
		\end{aligned}
	\]
if \underline{$a$ is (WD)}, and
	\[
		\begin{aligned}
			H^2_{a,0}(0,1):&=\{u\in H^2_a(0,1): u'(1)=0\}\\
			&=\{u\in H^1_0(0,1): u' \text{ is locally absolutely continuous in (0,1]},\\
			& \quad \quad u'(1)=0, \sqrt{a}u''\in L^2(0,1)\}
	\end{aligned}\]
if \underline{$a$ is (SD)}. On $H^2_{a,0}(0,1)$ we consider the same norms of $H^2_a(0,1)$ and the following one 
\[
			\|u\|_{2, \sim }:= \|\sqrt{a}u''\|_{L^2(0,1)} \quad\forall\; u\in \Ho. 
			\]

			\begin{Proposition}\label{norms} Assume $a$ (WD) or (SD). Then the norms $\|\cdot\|_{2, a}$, $\| \cdot\|_2$ and
			$
			\|\cdot\|_{2, \sim}
			$
			are equivalent in $\Ho$. In particular, 
			\begin{equation}\label{stimau'}
			\|u\|^2_{L^2(0,1)} \le\|u'\|^2_{L^2(0,1)} \le  \frac{1}{a(1)(2-K)}\|\sqrt{a}u''\|^2_{L^2(0,1)}, 
			\end{equation}
			for all $u \in \Ho$.
		\end{Proposition}
		\begin{proof} Clearly $\|u\|_{2, \sim } \le \|u \|_{2}\le \|u\|_{2, a}$ for all $u \in H^2_a(0,1)$ and, in particular, for all $u \in \Ho$. Thus, it is sufficient to prove
		that there exists a positive constant $C$ such that
		\[
		\|u \|_{2} \le C\|u\|_{2, \sim } \quad \text{or } \quad \|u\|_{2, a} \le C\|u\|_{2, \sim },
		\]
for all $u \in \Ho$.		
		Thus, fix $u\in \Ho$; since $u(0)=0$, we have
\begin{equation}\label{star}
\begin{aligned}
	|u(x)|&= \Biggl |\int_0^xu'(s)ds\Biggr |\le  \int_0^x	|u'(s)|ds\le \sqrt{x} \|u'\|_{L^2(0,1)}\le\|u'\|_{L^2(0,1)},
	\end{aligned}
\end{equation}
for all $x \in (0,1]$, 
and
\begin{equation}\label{star1}
\|u\|^2_{L^2(0,1)}\le\|u'\|^2_{L^2(0,1)}.
\end{equation}
Thus, it remains to estimate $ \|u'\|^2_{L^2(0,1)}$. To this aim, using the fact that $u'(1)=0$,  we rewrite
		\begin{equation*}
			|u'(x)|=\Biggl |\int_x^1\frac{\sqrt{a(t)}u''(t)}{\sqrt{a(t)}}dt \Biggr |\le 	\|\sqrt{a}u''\|_{L^2(0,1)}\Biggl (\int_x^1\frac{1}{a(t)}dt\Biggr )^{\frac{1}{2}},
		\end{equation*}
		for every $x\in (0,1]$.
	Thus,
	\begin{equation}\label{virizion}
		\begin{aligned}
			\|u'\|^2_{L^2(0,1)}&\le \|\sqrt{a}u''\|^2_{L^2(0,1)}\int_0^1\int_x^1\frac{1}{a(t)}dt\,dx\\
			&=\|\sqrt{a}u''\|^2_{L^2(0,1)}\int_0^1\frac{1}{a(t)}\int_0^tdx\,dt=\|\sqrt{a}u''\|^2_{L^2(0,1)}\int_0^1\frac{t}{a(t)}dt.
		\end{aligned}
	\end{equation}
By \eqref{crescente}, one has
\begin{equation}\label{2.12'}
	\int_0^1\frac{t}{a(t)}dt= \int_0^1\frac{t^K}{a(t)}t^{1-K}dt\le \frac{1}{a(1)}\int_0^1t^{1-K}dt=  \frac{1}{a(1)(2-K)}.
\end{equation}
Consequently
\[
\|u\|^2_{L^2(0,1)} \le \|u'\|^2_{L^2(0,1)} \le  \frac{1}{a(1)(2-K)}\|\sqrt{a}u''\|^2_{L^2(0,1)}
\]
and the thesis follows.
		\end{proof}

To complete the description of the functional setting needed to treat problem (\ref{(P_1)}), another important Hilbert space related to its well posedness is given by the following one:
	\begin{equation*}
		\mathcal{H}_0:=H^2_{a,0}(0,1)\times L^2(0,1),
	\end{equation*}
	endowed with inner product and norm given by
	\begin{equation*}
		\langle (u,v),(\tilde{u},\tilde{v})\rangle_{\mathcal{H}_0}:=\int_{0}^{1}au''\tilde{u}''dx+\int_{0}^{1}v\tilde{v}\,dx
	\end{equation*}
 and 
	\begin{equation*}
		\|(u,v)\|^2_{\mathcal{H}_0}:=\int_{0}^{1}a(u'')^2dx+\int_{0}^{1}v^2dx
	\end{equation*}
		for every $(u,v), (\tilde{u},\tilde{v})\in\mathcal{H}_0$. Now, consider the operators
$A: D(A)\subset L^2(0,1) \rightarrow L^2(0,1)$ and   $\mathcal{A}:D(\mathcal{A})\subset\mathcal{H}_0\to \mathcal{H}_0$ given by  
	\[
		Au:=(au'')'',
\]
	for all $u \in D(A)$, where
	\[
	\begin{aligned}
D(A):= \{ u \in \mathcal Z (0,1): &\; u'(0)=u'(1)=0, \text{ if $a$ is (WD), or} \\
	& (au'')(0)=u'(1)=0, \text{  if $a$ is (SD)}\},
		\end{aligned}
	\]
	and 
	\[	\mathcal{A}:=\begin{pmatrix}
		0 & Id \\
		-A & 0
	\end{pmatrix},\quad 
	\begin{aligned}
	D(\mathcal{A}):=D(A)\times  \Ho.
	\end{aligned}
	\]
	Observe that   for all $(u,v)\in D(\mathcal A)$ 
	\begin{equation}\label{GF0}
			\int_{0}^{1}(au'')''v\,dx=\int_{0}^{1}au''v''dx.
		\end{equation}
		Indeed, thanks to \eqref{GG}, it is sufficient to prove that there exists $\lim_{\delta \rightarrow 0^+}(au''v')(\delta)=0$. This is clearly satisfied if $a$ is (WD), since in this case $v$ is absolutely continuous in $[0,1]$ and $v'(0)=0$. Now, assume that $a$ is (SD). Since $(au''v')(1)=0$, one can write
		\[
		(au''v')(\delta)=-\int_\delta^1 (au''v')'(x)dx = -\int_\delta^1 (au'')'(x)v'(x)dx - \int_\delta^1 (au''v'')(x)dx,
		\]
		where $\delta \in (0,1)$ is fixed.
	Clearly, $au'' \in H^2(0,1)$, $\sqrt{a}u'', \sqrt{a}v'', v' \in L^2(0,1)$, thus $(au'')'v'$ and $au''v'' \in L^1(0,1)$ and, by the absolute continuity of the integral,  $\exists \; \lim_{\delta \rightarrow 0}(au''v')(\delta)=L \in \R.
$ Since $(u, v) \in D(\mathcal A)$, we know that $(au'')(0)=0$, hence \eqref{tondino} holds and
		\[
		|(au'')(x)|\le \|(au'')'\|_{L^2(0,1)}\sqrt{x}.
		\]
		This implies that, if $L \neq 0$, proceeding as in \eqref{stimatb} and \eqref{tondino1}, one has  $v' \not\in L^2(0,1)$; hence $L=0$.

	Thanks to \eqref{GF0} and using the semigroup technique, we can prove the following generation theorem whose proof is written in the Appendix.
\begin{Theorem}\label{th generazione}
	Assume $a$ (WD) or (SD). Then the operator $(\mathcal{A}, D(\mathcal{A}))$ is non positive with dense domain and generates a contraction semigroup $(T(t))_{t\ge 0}$.
\end{Theorem}	
	Thanks to the operator $(\mathcal{A},D(\mathcal{A}))$ and setting

	\[	\mathcal{U}(t):=\begin{pmatrix}
		y(t) \\
		y_t(t)
	\end{pmatrix}\quad\,\,\,\,\text{and}\quad\,\,\,\,\mathcal{U}_0:=\begin{pmatrix}
		y^0_T \\
		y^1_T
	\end{pmatrix},
	\] (\ref{(P_1)}) can be formulated as the  Cauchy problem
	\begin{equation}\label{Cauchy problem}
		\begin{cases}
			\dot{\mathcal{U}}(t)=\mathcal{A}\,\mathcal{U}(t), &t\ge 0,\\
			\mathcal{U}(0)=\mathcal{U}_0.
		\end{cases}
	\end{equation}
Thus, by Theorem \ref{th generazione}, if $\mathcal{U}_0\in\mathcal{H}_0$, the mild solution of (\ref{Cauchy problem}) is given by $\mathcal{U}(t)=T(t)\mathcal{U}_0$; if $\mathcal{U}_0\in D(\mathcal{A})$, then the solution $\mathcal{U}$ is more regular and the equation in (\ref{(P_1)}) holds for all $t\ge 0$. In particular, as in \cite[Proposition 3.15]{daprato}, one has the following theorem.
	\begin{Theorem}\label{Theorem 2.6}
		Assume $a$ (WD) or (SD).
		If $(y^0_T,y^1_T)\in\mathcal{H}_0$, then there exists a unique mild solution
		\begin{equation*}
			y\in \mathcal{C}^1([0,+\infty);L^2(0,1))\cap \mathcal{C}([0,+\infty);\Ho)
		\end{equation*}
		of (\ref{(P_1)}) which depends continuously on the initial data $(y^0_T,y^1_T)\in \mathcal{H}_0$. Moreover, if $(y^0_T,y^1_T)\in D(\mathcal{A})$, then the solution $y$ is classical, in the sense that
		\begin{equation*}
			y\in \mathcal{C}^2([0,+\infty);L^2(0,1))\cap \mathcal{C}^1([0,+\infty);\Ho)\cap \mathcal{C}([0,+\infty);D(A))
		\end{equation*}
		and the equation of (\ref{(P_1)}) holds for all $t\ge 0$.
	\end{Theorem}
	\begin{Remark}\label{remark revers}
		Due to the reversibility (in time) of the differential equation, solutions exist with the same regularity also for $t<0$.
	\end{Remark}

	\subsection{Energy estimates}
	In this subsection we prove some estimates from below and from above of the energy associated to the solution of \eqref{(P_1)} in order to obtain an observability inequality. First of all we give the  definition of the energy.
\begin{Definition}
		Let $y$ be a mild solution of (\ref{(P_1)}) and consider its energy given by the continuous function defined as
	\[
		E_y(t):=\frac{1}{2}\int_0^1 \Bigl (y^2_t(t,x)+a(x)y^2_{xx}(t,x) \Bigr )dx\quad\,\,\,\,\,\,\forall\;t\ge 0.
	\]
\end{Definition}
	The definition above is natural and guarantees that the classical conservation of the energy still holds true also in this degenerate situation.
	
	\begin{Theorem}\label{teorema energia costante}
		Assume $a$ (WD) or (SD) and let $y$ be a mild solution of (\ref{(P_1)}). Then
		\[
			E_y(t)=E_y(0)\quad\,\,\,\,\,\,\,\forall\;t\ge 0.
		\]
	\end{Theorem}
	\begin{proof}
		First of all assume that $y$ is a classical solution. Then multiplying the equation
		\begin{equation*}
			y_{tt}+(ay_{xx})_{xx}=0
		\end{equation*}
		by $y_t$, integrating over $(0,1)$ and using the formula of integration by parts (\ref{GF0}), one has
		\begin{equation*}
			\begin{aligned}
				0&=\frac{1}{2}\int_0^1 \Bigl (y^2_t\Bigr )_tdx+\int_0^1 (ay_{xx})_{xx}\,y_tdx\\
				&=\frac{1}{2}\frac{d}{dt}\Biggl (\int_0^1\Bigl (y^2_t+ay^2_{xx}\Bigr )dx\Biggr )=\frac{d}{dt}E_y(t).
			\end{aligned}
		\end{equation*}
		Consequently the energy associated to $y$ is constant.	
If $y$ is a mild solution, we can proceed as in \cite{CF_Beam}.
	\end{proof}

	Now we prove an inequality for the energy which we will use in the next subsection to establish the controllability result. First of all, we start proving the following preliminary result.
	\begin{Theorem}\label{Teorema 1}
		Assume $a$ (WD) or (SD). If $y$ is a classical solution of (\ref{(P_1)}), then $y_{xx}(\cdot,1)\in L^2(0,T)$ for any $T>0$ and
		\begin{equation}\label{prima uguaglianza0}
			\begin{aligned}
				\frac{1}{2}a(1)\int_0^T y^2_{xx}(t,1)dt&=\int_0^1\Bigl [y_tx^2y_x \Bigr ]^{t=T}_{t=0}dx+\int_{Q_T}xy^2_tdx\,dt\\
				&+\int_{Q_T}\Bigl (3xa-\frac{x^2}{2}a'\Bigr )y^2_{xx}dx\,dt-\int_{Q_T}a'y^2_xdx\,dt.
			\end{aligned}
		\end{equation}
	\end{Theorem}
	\begin{proof}
		Multiplying the equation in (\ref{(P_1)}) by $x^2y_x$ and integrating over $Q_T=(0,T)\times (0,1)$, we obtain
		\[
			\begin{aligned}
				0&=\int_{Q_T}y_{tt}x^2y_xdx\,dt+\int_{Q_T}x^2y_x(ay_{xx})_{xx}dx\,dt\\
				&=\int_0^1\Bigl [y_tx^2y_x \Bigr ]^{t=T}_{t=0}dx-\frac{1}{2}\int_{Q_T}x^2(y^2_t)_xdx\,dt+\int_{Q_T}x^2y_x(ay_{xx})_{xx}dx\,dt\\
				&=\int_0^1\Bigl [y_tx^2y_x \Bigr ]^{t=T}_{t=0}dx-\frac{1}{2}\int_0^T\Bigl [x^2y^2_t \Bigr ]^{x=1}_{x=0}dt+\int_{Q_T}x y^2_tdx\,dt+\int_{Q_T}x^2y_x(ay_{xx})_{xx}dx\,dt\\
				&=\int_0^1\Bigl [y_tx^2y_x \Bigr ]^{t=T}_{t=0}dx+\int_{Q_T}x y^2_tdx\,dt+\int_{Q_T}x^2y_x(ay_{xx})_{xx}dx\,dt
			\end{aligned}
		\]
		since, thanks to  the boundary conditions of $y$, we have $y_t(t,0)=0= y_t(t,1)$.
Now, consider the term $\int_{Q_T}x^2y_x(ay_{xx})_{xx}dx\,dt$, which is clearly well defined, and let $\delta >0$. Then
		\begin{equation}\label{secondo integrale}
			\int_{Q_T}x^2y_x(ay_{xx})_{xx}dx\,dt=\int_0^T\int_0^\delta x^2y_x(ay_{xx})_{xx}dx\,dt+\int_0^T\int_\delta^1x^2y_x(ay_{xx})_{xx}dx\,dt.
		\end{equation}
	By the H\"older's inequality, $x^2y_x (ay_{xx})_{xx} \in L^1(0,1)$, hence
		\begin{equation*}
			\lim_{\delta\to 0}\int_0^T\int_0^\delta x^2y_x(ay_{xx})_{xx}dx\,dt= 0,
		\end{equation*}
		by the absolute continuity of the integral. In the next steps, we will estimate the second term of (\ref{secondo integrale}). Setting $I:= [\delta, 1]$, since $y \in D(A)$, we have $ay_{xx}\in H^2(I)$ and, in particular, $y_{xx} \in H^2(I)$ since $a(x) \neq 0$ for all $x \in I$. Hence, we can integrate by parts
		\begin{equation}\label{espressione con BT}
			\begin{aligned}
				&\int_0^T\int_\delta^1x^2y_x(ay_{xx})_{xx}dx\,dt=\int_0^T[x^2y_x(ay_{xx})_{x}]^{x=1}_{x=\delta}dt-\int_0^T\int_\delta^1(x^2y_x)_x(ay_{xx})_{x}dx\,dt\\
				&=\int_0^T[x^2y_x(ay_{xx})_{x}]^{x=1}_{x=\delta}dt-2\int_0^T\int_\delta^1xy_x(ay_{xx})_{x}dx\,dt-\int_0^T\int_\delta^1 x^2y_{xx}(ay_{xx})_{x}dx\,dt\\
				&=\int_0^T[x^2y_x(ay_{xx})_{x}]^{x=1}_{x=\delta}dt-2\int_0^T[xy_xay_{xx}]^{x=1}_{x=\delta}dt+2\int_0^T\int_\delta^1(xy_x)_xay_{xx}dx\,dt\\
				&-\int_0^T[x^2ay^2_{xx}]^{x=1}_{x=\delta}dt+\int_0^T\int_\delta^1ay_{xx}(x^2y_{xx})_xdx\,dt\\
				&=\int_0^T[x^2y_x(ay_{xx})_{x}]^{x=1}_{x=\delta}dt-2\int_0^T[xy_xay_{xx}]^{x=1}_{x=\delta}dt-\int_0^T[x^2ay^2_{xx}]^{x=1}_{x=\delta}dt\\
				&+2\int_0^T\int_\delta^1ay_xy_{xx}dx\,dt+2\int_0^T\int_\delta^1xay^2_{xx}dx\,dt+2\int_0^T\int_\delta^1xay^2_{xx}dx\,dt\\
				&+\int_0^T\int_\delta^1x^2ay_{xx}y_{xxx}dx\,dt\\
				&=\int_0^T[x^2y_x(ay_{xx})_{x}]^{x=1}_{x=\delta}dt-2\int_0^T[xy_xay_{xx}]^{x=1}_{x=\delta}dt-\int_0^T[x^2ay^2_{xx}]^{x=1}_{x=\delta}dt\\
				&+4\int_0^T\int_\delta^1xay^2_{xx}dx\,dt+\int_0^T\int_\delta^1a(y^2_x)_xdx\,dt+\frac{1}{2}\int_0^T\int_\delta^1x^2a(y^2_{xx})_xdx\,dt
\end{aligned}
		\end{equation}
\[
\begin{aligned}
				&=\int_0^T[x^2y_x(ay_{xx})_{x}]^{x=1}_{x=\delta}dt-2\int_0^T[xy_xay_{xx}]^{x=1}_{x=\delta}dt-\int_0^T[x^2ay^2_{xx}]^{x=1}_{x=\delta}dt\\
				&+4\int_0^T\int_\delta^1xay^2_{xx}dx\,dt
				+\int_0^T[ay^2_{x}]^{x=1}_{x=\delta}dt-\int_0^T\int_\delta^1a'y^2_xdx\,dt\\
				&+\frac{1}{2}\int_0^T[x^2ay^2_{xx}]^{x=1}_{x=\delta}dt-\frac{1}{2}\int_0^T\int_\delta^1(2xa+x^2a')y^2_{xx}dx\,dt\\
				&=\int_0^T[x^2y_x(ay_{xx})_{x}]^{x=1}_{x=\delta}dt-2\int_0^T[xy_xay_{xx}]^{x=1}_{x=\delta}dt-\frac{1}{2}\int_0^T[x^2ay^2_{xx}]^{x=1}_{x=\delta}dt\\
				&+\int_0^T[ay^2_{x}]^{x=1}_{x=\delta}dt
				-\int_0^T\int_\delta^1a'y^2_xdx\,dt+\int_0^T\int_\delta^1\Bigl (3xa-\frac{x^2}{2}a' \Bigr )y^2_{xx}dx\,dt.
			\end{aligned}
		\]
		Observe that in the strongly degenerate case  $\int_0^T\int_0^1a'y^2_xdx\,dt\in \R$, being $y_x \in L^2(0,1)$ and $a \in \mathcal C^1[0,1]$. In particular,
			by \eqref{stimau'}, we have
	\begin{equation}\label{quarto}
	\int_{Q_T}| a' y_x^2|dx\,dt \le \max_{x \in [0,1]}|a'(x)|\frac{1}{a(1)(2-K)} \int_{Q_T} ay_{xx}^2dx\,dt.
	\end{equation}
	In the weakly degenerate case, it results
	\[
	\int_{Q_T}| a' y_x^2|dx\,dt \le \int_{Q_T}\frac{1}{x^2}x^2 |a'| y_x^2dx\,dt  \le K\int_{Q_T} \frac{a}{x^2} y_x^2dx\,dt.
	\]
Thus, we can apply Proposition \ref{HP} (where $\theta =K \in (0,1)$) to $y_x$, obtaining
	\begin{equation}\label{quartoWD}
	\int_{Q_T}| a' y_x^2|dx\,dt  \le K\int_{Q_T} \frac{a}{x^2} y_x^2dx\,dt \le \frac{4K}{(1-K)^2} \int_{Q_T} ay_{xx}^2dx\,dt.
	\end{equation}
	In every case $\int_0^T\int_0^1a'y^2_xdx\,dt\in \R$ and
	\[
	\int_{Q_T}| a' y_x^2|dx\,dt  \le \max\left\{\frac{\max_{x \in [0,1]}|a'(x)|}{a(1)(2-K)}, \frac{4K}{(1-K)^2}\right\} \int_{Q_T} ay_{xx}^2dx\,dt.
	\]
			Now, consider the term $\ds\int_0^T\int_\delta^1\Bigl (3xa-\frac{x^2}{2}a' \Bigr )y^2_{xx}dx\,dt$.  Clearly, $xay^2_{xx}\in L^1(0,1)$; moreover 
		\[
	\left|	\frac{x^2}{2}a' y_{xx}^2\right| \le \frac{K}{2}ax  y_{xx}^2 \in L^1(0,1). 
		\]
	Thus, by the absolute continuity of the integral, 
	\[\displaystyle\lim_{\delta\to 0}\int_0^T\int_\delta ^1\Bigl (3xa-\frac{x^2}{2}a' \Bigr )y^2_{xx}dx\,dt= \int_0^T\int_0^1 \Bigl (3xa-\frac{x^2}{2}a' \Bigr )y^2_{xx}dx\,dt.\]
		It remains to estimate the boundary terms in (\ref{espressione con BT}). To this aim observe that, thanks to the boundary conditions of $y$,
		\begin{equation}\label{BT}
			\begin{aligned}
				&[x^2y_x(ay_{xx})_{x}]^{x=1}_{x=\delta}-2[xy_xay_{xx}]^{x=1}_{x=\delta}-\frac{1}{2}[x^2ay^2_{xx}]^{x=1}_{x=\delta}+[ay^2_{x}]^{x=1}_{x=\delta}\\
				&=-\delta^2y_x(t,\delta)(ay_{xx})_x(t,\delta)+2\delta y_x(t,\delta)(ay_{xx})(t,\delta)-\frac{1}{2}(ay^2_{xx})(t,1)\\&+\frac{1}{2}\delta^2(ay^2_{xx})(t,\delta)-a(\delta)y^2_x(t,\delta).
			\end{aligned}
		\end{equation}
	 Thanks to Lemma \ref{boundaryconditions_WD}, if $a$ is (WD), or Lemma \ref{boundaryconditions_SD}, if $a$ is (SD), one has that the boundary terms in \eqref{BT} reduce to $-\ds\frac{1}{2}(ay^2_{xx})(t,1)$.		
	Hence, by \eqref{espressione con BT}, we have
		\begin{equation*}
		\begin{aligned}
			\lim_{\delta\to 0}\int_0^T\int_\delta^1x^2y_x(ay_{xx})_{xx}dx\,dt&=-\frac{1}{2}a(1)\int_0^Ty^2_{xx}(t,1)dt-\int_0^T\int_0^1a'y^2_xdx\,dt\\
			& +\int_0^T\int_0^1\Bigl (3xa-\frac{x^2}{2}a' \Bigr )y^2_{xx}dx\,dt
			\end{aligned}
		\end{equation*}
		and (\ref{prima uguaglianza0}) follows.
	\end{proof}

Thanks to the previous equality and Proposition \ref{HP} one can prove an  estimate from below on the energy using the term $\displaystyle \int_0^T y^2_{xx}(t,1)dt$.

\begin{Theorem}
	Assume $a$ (WD) or (SD). If $y$ is a mild solution of (\ref{(P_1)}), then
	\[
			a(1)\int_0^Ty^2_{xx}(t,1)dt\le 4\left(T\max\left\{1, \frac{6+K}{2}+ \frac{4K}{(1-K)^2}\right\}+\max\Biggl \{
\frac{1}{a(1)(2-K)},1 \Biggr \}\right) E_y(0),
	\]
	if $a$ is (WD), and
	\[
			a(1)\int_0^Ty^2_{xx}(t,1)dt\le 4\left(T\max\left\{1, \frac{6+K}{2}+\frac{\max_{x \in [0,1]}|a'(x)| }{a(1)(2-K)} \right\}+\max\Biggl \{
\frac{1}{a(1)(2-K)},1 \Biggr \}\right) E_y(0),
	\]
	if $a$ is (SD).
\end{Theorem}
\begin{proof}
As a first step we estimate the distributed terms in \eqref{prima uguaglianza0}.
 Using the fact that $x|a'|\le Ka$, by (\ref{crescente}), we find
	\begin{equation}\label{secondo}
		\begin{aligned}
			\Biggl |\int_{Q_T}\Bigl (3xa-\frac{x^2a'}{2}\Bigr )y^2_{xx}dx\,dt \Biggr |&\le \frac{6+K}{2}\int_{Q_T}xay^2_{xx}dx\,dt\le \frac{6+K}{2}\int_{Q_T}ay^2_{xx}dx\,dt.
		\end{aligned}
	\end{equation}
	Clearly, 
	\begin{equation}\label{terzo}
		\int_{Q_T}xy^2_{t}dx\,dt\le \int_{Q_T}y^2_{t}dx\,dt.
	\end{equation}
	Now, we will estimate the boundary terms in \eqref{prima uguaglianza0}. By \eqref{stimau'}, one has
\[
		\begin{aligned}
\Biggl |\int_0^1x^2y_x(\tau,x)y_t(\tau,x)dx \Biggr |&\le\frac{1}{2}\int_0^1x^4y^2_x(\tau,x)dx+\frac{1}{2}\int_0^1y^2_t(\tau,x)dx\\
&\le \frac{1}{2}
\frac{1}{a(1)(2-K)} \int_0^1 a(x)y_{xx}^2(\tau,x)dx\,dt+\frac{1}{2}\int_0^1y^2_t(\tau,x)dx
\end{aligned}
	\]
	for all $\tau\in [0,T]$. 
	By Theorem \ref{teorema energia costante}, we get
	\begin{equation}\label{primo}
		\begin{aligned}
			\Biggl |\int_0^1\Bigl [x^2y_x(\tau,x)y_t(\tau,x) \Bigr ]^{\tau =T}_{\tau =0}dx\Biggr |&\le  \frac{1}{2}
\frac{1}{a(1)(2-K)}\int_0^1a(x)y^2_{xx}(T,x)dx+\frac{1}{2}\int_0^1y^2_t(T,x)dx\\
			&+ \frac{1}{2}
\frac{1}{a(1)(2-K)}\int_0^1a(x)y^2_{xx}(0,x)dx+\frac{1}{2}\int_0^1y^2_t(0,x)dx\\
			&\le 2\max\Biggl \{
\frac{1}{a(1)(2-K)},1 \Biggr \}E_y(0).
		\end{aligned}
	\end{equation}
	By \eqref{prima uguaglianza0}, \eqref{quarto}, \eqref{quartoWD}, (\ref{secondo}), (\ref{terzo}) and  (\ref{primo}), 
	we have
	\[
			\begin{aligned}
	\frac{1}{2}a(1)\int_0^Ty_{xx}^2(t,1)dt &\le \int_{Q_T}y_t^2dxdt + \left(\frac{6+K}{2}+ 
	\frac{4K}{(1-K)^2}\right)\int_{Q_T}ay_{xx}^2dx\,dt \\
	&+2\max\Biggl \{
\frac{1}{a(1)(2-K)},1 \Biggr \}E_y(0),
	\end{aligned}
	\]
	in the weakly degenerate case,
	and
		\[
			\begin{aligned}
	\frac{1}{2}a(1)\int_0^Ty_{xx}^2(t,1)dt &\le \int_{Q_T}y_t^2dxdt + \left(\frac{6+K}{2}+\frac{\max_{x \in [0,1]}|a'(x)| }{a(1)(2-K)}\right)\int_{Q_T}ay_{xx}^2dx\,dt \\
	&+2\max\Biggl \{
\frac{1}{a(1)(2-K)},1 \Biggr \}E_y(0),
	\end{aligned}
	\]
	in the strongly one. Thus, the thesis follows if $y$ is a classical solution; on the other hand, if $y$ is a mild solution, we can proceed as in \cite{CF_Beam}.
\end{proof}
	
In order to obtain an estimate from above on the energy, the next preliminary result is useful.
	\begin{Theorem}
		Assume $a$ (WD) or (SD). If $y$ is a classical solution of (\ref{(P_1)}), then $y_{xx}(\cdot,1)\in L^2(0,T)$ for any $T>0$ and
		\begin{equation}\label{seconda uguaglianza}
			\begin{aligned}
				\frac{1}{2}a(1)\int_0^T y^2_{xx}(t,1)dt&=\int_0^1[xy_ty_x ]^{t=T}_{t=0}dx+\frac{1}{2}\int_{Q_T}y^2_tdx\,dt+\frac{1}{2}\int_{Q_T}(3a-xa')y^2_{xx}dx\,dt.
			\end{aligned}
		\end{equation}
	\end{Theorem}
	\begin{proof}
		Multiplying the equation in (\ref{(P_1)}) by $xy_x$, integrating over $Q_T$ and using the boundary conditions of $y$, we obtain
		\begin{equation}\label{uguaglianza 2}
			\begin{aligned}
				0&=\int_0^1[xy_xy_t ]^{t=T}_{t=0}dx-\frac{1}{2}\int_{Q_T}x(y^2_t)_xdx\,dt+\int_{Q_T}xy_x(ay_{xx})_{xx}dx\,dt\\
						&		= \int_0^1[xy_xy_t ]^{t=T}_{t=0}dx+\frac{1}{2}\int_{Q_T} y^2_tdx\,dt+\int_{Q_T}xy_x(ay_{xx})_{xx}dx\,dt.
			\end{aligned}
		\end{equation}
Now, let $\delta >0$ and rewrite the well defined integral $\int_{Q_T}xy_x(ay_{xx})_{xx}dxdt$ as
		\begin{equation*}
			\int_{Q_T}xy_x(ay_{xx})_{xx}dx\,dt=\int_0^T\int_0^\delta xy_x(ay_{xx})_{xx}dx\,dt+\int_0^T\int_\delta^1 xy_x(ay_{xx})_{xx}dx\,dt.
		\end{equation*}
	Due to the absolute continuity of the integral we get
		\begin{equation*}
			\lim_{\delta\to 0}\int_0^T\int_0^\delta xy_x(ay_{xx})_{xx}dx\,dt=0.
		\end{equation*}
		Moreover, it is possible to integrate by parts the second term:
		\begin{equation}\label{Natale}
			\begin{aligned}
				\int_0^T\int_\delta^1 xy_x(ay_{xx})_{xx}dx\,dt&=\int_0^T\Bigl [xy_x(ay_{xx})_{x}\Bigr ]^{x=1}_{x=\delta }dt-\int_0^T\int_\delta^1(xy_x)_x(ay_{xx})_{x}dx\,dt\\
				&=-\int_0^T\delta y_x(t,\delta)(ay_{xx})_{x}(t,\delta)dt-\int_0^T\int_\delta^1y_x(ay_{xx})_{x}dx\,dt\\
				&-\int_0^T\int_\delta^1xy_{xx}(ay_{xx})_{x}dx\,dt\\
				&=-\int_0^T\delta y_x(t,\delta)(ay_{xx})_{x}(t,\delta)dt-\int_0^T[ay_xy_{xx}]^{x=1}_{x=\delta }dt\\
				&+\int_0^T\int_\delta^1ay^2_{xx}dx\,dt-\int_0^T[axy^2_{xx}]^{x=1}_{x=\delta }dt
				+\int_0^T\int_\delta^1(xy_{xx})_{x}ay_{xx}dx\,dt\\
				&=-\int_0^T\delta y_x(t,\delta)(ay_{xx})_{x}(t,\delta)dt+\int_0^Ta(\delta)y_x(t, \delta)y_{xx}(t, \delta)dt\\
				&-\int_0^T[axy^2_{xx}]^{x=1}_{x=\delta }dt+\int_0^T\int_\delta^1ay^2_{xx}dx\,dt
				+\int_0^T\int_\delta^1ay^2_{xx}dx\,dt\\
				&+\frac{1}{2}\int_0^T\int_\delta^1xa(y^2_{xx})_xdx\,dt\\
				&=-\int_0^T\delta y_x(t,\delta)(ay_{xx})_{x}(t,\delta)dt+\int_0^Ta(\delta)y_x(t, \delta)y_{xx}(t, \delta)dt\\
				&-\frac{1}{2}\int_0^T[axy^2_{xx}]^{x=1}_{x=\delta }dt+\frac{1}{2}\int_0^T\int_\delta^1(3a-xa')y^2_{xx}dx\,dt.
			\end{aligned}
		\end{equation}
	As in the proof of Theorem \ref{Teorema 1}, one has
	\begin{equation*}
\lim_{\delta\to 0}\int_0^T\int_\delta^1(3a-xa')y^2_{xx}dx\,dt=\int_{Q_T}(3a-xa')y^2_{xx}dx\,dt
	\end{equation*}
	and, thanks to Lemma \ref{boundaryconditions_WD} or \ref{boundaryconditions_SD}, the boundary terms in \eqref{Natale} reduce to $- \ds \frac{1}{2}(ay^2_{xx})(t,1)$.
	
	Hence by the previous limit, \eqref{uguaglianza 2} and (\ref{Natale}) we obtain the thesis.
	\end{proof}
	
	Thanks to the equality stated in the previous theorem, one can prove an estimate from above on the energy. This estimate will be crucial to obtain the observability of \eqref{(P_1)}.
	\begin{Theorem}\label{Theorem OI}
			Assume $a$ (WD) or (SD). If $y$ is a mild solution of (\ref{(P_1)}), then
		\[
			a(1)\int_0^T y^2_{xx}(t,1)dt\ge \Biggl (T\min\left\{6-3K,K+2\right\}- (K+4)\max\Biggl \{
\frac{1}{a(1)(2-K)},1 \Biggr \}\Biggr )E_y(0)
		\]
		for any $T>0$.
	\end{Theorem}
	\begin{proof}
As a first step, assume that $y$ is a classical solution of (\ref{(P_1)}). Then, multiplying the equation in (\ref{(P_1)}) by $\displaystyle\frac{-Ky}{2}$, integrating by parts over $Q_T$ and applying (\ref{GF0}), we have
		\begin{equation*}
			\begin{aligned}
				0&=-\frac{K}{2}\int_{Q_T}y_{tt}y\,dx\,dt-\frac{K}{2}\int_{Q_T}y(ay_{xx})_{xx}dx\,dt\\
				&=-\frac{K}{2}\int_0^1[yy_t ]^{t=T}_{t=0}dx+\frac{K}{2}\int_{Q_T}y^2_tdx\,dt-\frac{K}{2}\int_{Q_T}ay^2_{xx}dx\,dt.
			\end{aligned}
		\end{equation*}
		Summing the previous equality to (\ref{seconda uguaglianza}) multiplied by $2$ and using  (\ref{sup}), we have
		\begin{equation*}
			\begin{aligned}
				a(1)\int_0^Ty^2_{xx}(t,1)dt&=2\int_0^1[xy_ty_x ]^{t=T}_{t=0}dx-\frac{K}{2}\int_0^1 [yy_t ]^{t=T}_{t=0}dx+\Bigl (\frac{K}{2}+1\Bigr )\int_{Q_T}y^2_tdx\,dt\\
				&+\int_{Q_T}\Bigl (3-\frac{xa'}{a}-\frac{K}{2}\Bigr )ay^2_{xx}dx\,dt\\
				& \ge2\int_0^1[xy_ty_x ]^{t=T}_{t=0}dx-\frac{K}{2}\int_0^1 [yy_t ]^{t=T}_{t=0}dx+\Bigl (\frac{K}{2}+1\Bigr )\int_{Q_T}y^2_tdx\,dt\\
				&+\int_{Q_T}\Bigl (3-\frac{3K}{2}\Bigr )ay^2_{xx}dx\,dt\\
				& \ge2\int_0^1[xy_ty_x ]^{t=T}_{t=0}dx-\frac{K}{2}\int_0^1 [yy_t ]^{t=T}_{t=0}dx\\
				&+\min\left\{6-3K,K+2\right\}\frac{1}{2}\int_{Q_T}(y^2_t+ ay^2_{xx})dx\,dt.
			\end{aligned}
		\end{equation*}
	Now, we analyze the boundary terms that appear in the previous relation. As in \eqref{primo}, one has 
		
		\[
		\begin{aligned}
		2	\Biggl |\int_0^1\Bigl [xy_x(\tau,x)y_t(\tau,x) \Bigr ]^{\tau =T}_{\tau =0}dx\Biggr |
			&\le 4\max\Biggl \{
\frac{1}{a(1)(2-K)},1 \Biggr \}E_y(0).
		\end{aligned}
	\]
		Furthermore, by Proposition \ref{norms}, Theorem \ref{teorema energia costante}  and \eqref{stimau'},   we obtain
		\begin{equation*}
			\begin{aligned}
				\Biggl |\int_0^1y(\tau,x)y_t(\tau,x)dx \Biggr |
				&\le \frac{1}{2}\int_0^1y^2_t(\tau,x)dx+
				\frac{1}{2a(1)(2-K)}\int_0^1a(x)y_{xx}^2(\tau,x)dx\\
				&\le \max\Biggl \{1,\frac{1}{a(1)(2-K)}\Biggr \}E_y(\tau)=\max\Biggl \{1,\frac{1}{a(1)(2-K)}\Biggr \}E_y(0)
			\end{aligned}
		\end{equation*}
		for all $\tau \in [0,T]$; hence
		\[
			\begin{aligned}
				\frac{K}{2}\left|\ds \int_0^1[yy_t ]^{\tau =T}_{\tau =0}dx \right| &\le K\max\Biggl \{1,\frac{1}{a(1)(2-K)}\Biggr \}E_y(0)
			\end{aligned}
		\]
		and the thesis follows if $y$ is a classical solution. If $y$ is a mild solution we proceed as in \cite{CF_Beam}.
	\end{proof}

	\subsection{Boundary observability and null controllability}
	As in \cite{CF_Beam} we give the following definition.
	\begin{Definition}
		Problem (\ref{(P_1)}) is said to be observable in time $T>0$ via the second derivative at $x=1$ if there exists a constant $C>0$ such that for any $(y^0_T,y^1_T)\in \mathcal{H}_0$ the classical solution $y$ of (\ref{(P_1)}) satisfies
		\begin{equation}\label{3.21}
			CE_y(0)\le \int_0^Ty^2_{xx}(t,1)dt.
		\end{equation}
		Moreover, any constant satisfying (\ref{3.21}) is called observability constant for (\ref{(P_1)}) in time $T$.
	\end{Definition}
	Setting 
	\begin{equation*}
		C_T:=\sup\bigl \{C>0: \text{$C$ satisfies (\ref{3.21})}\bigr \},
	\end{equation*}
	we have that problem (\ref{(P_1)}) is observable if and only if 
	\begin{equation*}
		C_T=\inf_{(y^0_T,y^1_T)\neq (0,0)}\frac{\int_0^Ty^2_{xx}(t,1)dt}{E_y(0)}>0.
	\end{equation*}
	The inverse of $C_T$, i.e. $c_T:=\frac{1}{C_T}$, is called the {\it cost of observability} (or the {\it cost of control}) in time $T$.
	
	Thanks to Theorem \ref{Theorem OI} one can estimate the constant $C_T$, obtaining the observability of \eqref{(P_1)}.
	\begin{Corollary}\label{cor 3.12}
		Assume $a$  (WD)  or (SD). If 
		\begin{equation*}\label{ipo 3.22}
			T> \frac{K+4}{\min \{ 6-3K, K+2\}} \max\left\{1, \frac{1}{a(1)(2-K)}\right\} ,
		\end{equation*}
		then (\ref{(P_1)}) is observable in time $T$. Moreover
		\begin{equation*}
		C_T \ge T \min \{ 6-3K, K+2\} - (K+4) \max\left\{1, \frac{1}{a(1)(2-K)}\right\}.
		\end{equation*}
	\end{Corollary}

	Now, we give the definition of solution by transposition for the problem \eqref{(P)}.
	
	\begin{Definition}
		Let $f\in L^2_{loc}[0,+\infty)$ and $(u_0,u_1)\in L^2(0, 1)\times \left(H^{2}_{a,0}(0,1)\right)^*$. We say that $u$ is a solution by transposition of (\ref{(P)}) if
		\begin{equation*}
			u\in \mathcal{C}^1\left([0,+\infty);\left(H^{2}_{a,0}(0,1)\right)^*\right)\cap \mathcal{C}([0,+\infty);L^2(0,1))
		\end{equation*}
		and for all $T>0$
		\[
			\begin{aligned}
				\left\langle u_t(T),v^0_T\right\rangle_{\left(H^{2}_{a,0}(0,1)\right)^*,H^{2}_{a,0}(0, 1)}&-\int_0^1u(T,x)v^1_T\,dx=\left\langle u_1,v(0)\right\rangle_{\left(H^{2}_{a,0}(0,1)\right)^*,H^{2}_{a,0}(0, 1)}\\
				&-\int_0^1u_0(x)v_t(0,x)\,dx-\int_0^Ta(1)v_{xx}(t,1)f(t)dt
			\end{aligned}
		\]
for all $(v^0_T,v^1_T)\in \mathcal H_0$, where $v$ solves the backward problem
		\begin{equation}\label{backward problem}
			\begin{cases}
				v_{tt}(t,x)+(av_{xx})_{xx}(t,x)=0, &(t,x)\in (0,+\infty) \times (0,1),\\
				v(t,0)=0,\,\,v(t,1)=0,&t>0,\\
				\begin{cases}
					v_x(t,0)=0, &\text{ if } a \text{ is (WD)},\\
					(av_{xx})(t,0)=0, &\text{ if } a \text{ is (SD)},\\ 
				\end{cases}&t>0,\\
			 v_x(t,1)=0, &t>0,\\
				v(T,x)=v^0_T(x),\,\,v_t(T,x)=v^1_T(x),&x\in(0,1).
			\end{cases}
		\end{equation}
	\end{Definition}
	
	In order to study the null controllability for \eqref{(P)}, we consider the bilinear form $\Lambda:\mathcal{H}_0\times\mathcal{H}_0\to\mathbb{R}$ defined as
	\begin{equation*}
		\Lambda(V_T,W_T):=\int_0^Ta(1)v_{xx}(t,1)w_{xx}(t,1)dt,
	\end{equation*}
	where $v$ and $w$ are the solutions of (\ref{backward problem}) associated to the data $V_T:= (v^0_T,v^1_T)$ and $W_T:= (w^0_T,w^1_T)$, respectively. As in \cite{CF_Beam}, assuming $a$ (WD)  or (SD), one has that $\Lambda$ is continuous and coercive and, using these properties, one can show the null controllability for  (\ref{(P)}). Indeed, setting $T_0$ the lower bound found in Corollary \ref{cor 3.12}, i.e. 
	\[
	T_0:= \frac{K+4}{\min \{ 6-3K, K+2\}}  \max\left\{1, \frac{1}{a(1)(2-K)}\right\},
	\]
	one has the next result.
	\begin{Theorem}
	Assume $a$ (WD)  or (SD). Then, for all $T>T_0$ and for every $(u_0,u_1)\in L^2(0, 1)\times \left(H^{2}_{a,0}(0,1)\right)^*$, there exists a control $f\in L^2(0,T)$ such that the solution of (\ref{(P)}) satisfies
		\[
			u(T,x)=u_t(T,x)=0\,\,\,\,\,\,\forall\; x\in (0,1).
		\]
	\end{Theorem}
	We omit the proof since it can be easily obtained by the one of \cite[Theorem 4.1]{CF_Beam} with suitable changes.

	\section{The linear stabilization}\label{section 4}
	In this section we consider the second problem of the paper, i.e. the linear stabilization of the following degenerate beam equation
\begin{equation}\label{(P_feed)}
	\begin{cases}
		y_{tt}(t,x)+(ay_{xx})_{xx}(t,x)=0, &(t,x)\in Q_T,\\
			y(t,0)=0,&t\in (0,T),\\
				\begin{cases}
			y_x(t,0)=0, &\text{ if } a \text{ is (WD)},\\
		(ay_{xx})(t,0)=0, &\text{ if } a \text{ is (SD)},\\ 
			\end{cases}, &t\in (0, T),\\
		\beta y(t,1)-(ay_{xx})_x(t,1)+y_t(t,1)=0, &t \in (0,T),\\
		\gamma y_x(t,1)+(ay_{xx})(t,1)+y_{tx}(t,1)=0, &t \in (0,T),\\
		y(0,x)=y_0(x),\,\,y_t(0,x)=y_1(x),&x\in(0,1),
	\end{cases}
\end{equation}
where, as usual, $Q_T:=(0,T) \times (0,1)$, $T>0$ and $\beta,\gamma \ge0$.

\subsection{The well posedness}

In order to study the well posedness of (\ref{(P_feed)}), we introduce the following Hilbert spaces: if the function \underline{$a$ is (WD)},  we consider
	\[
		\begin{aligned}
			K^2_a(0,1):&=\{u \in V^2_a(0,1): u(0)=0\}\\
			&=\{u\in H^1(0,1): u' \text{ is absolutely continuous in [0,1]}, u(0)=0,\\
			& \quad \;\; \sqrt{a}u''\in L^2(0,1)\};
		\end{aligned}
	\]
on the other hand, if \underline{$a$ is (SD)}, then 
	\[
		\begin{aligned}
			K^2_a(0,1):&=\{u \in V^2_a(0,1): u(0)=0\}\\
			&=\{u\in H^1(0,1): u' \text{ is locally absolutely continuous in } (0,1], u(0)=0,\\
			& \quad \; \;\sqrt{a}u''\in L^2(0,1)\}.
	\end{aligned}\]
	On $K^2_a(0,1)$ we consider the norms
	$
		\|\cdot\|_{2,a}$ or
		$
		\|\cdot\|_{2}$ defined in \eqref{normadiv} and in \eqref{normadiv1}, respectively, and the following one
\[
\|u\|_{2, \circ}^2 :=  |u'(1)|^2+\|\sqrt{a}u''\|^2_{L^2(0,1)}
\]
for all $u \in K^2_{a}(0, 1)$.
	\begin{Proposition}\label{normeequivalenti} Assume $a$ (WD) or (SD). Then the three norms are equivalent in $K^2_{a}(0, 1)$. In particular,  for all $u\in K^2_a(0, 1)$
			\begin{equation}\label{stimaL2}
			\|u\|^2_{L^2(0,1)}\le\|u'\|^2_{L^2(0,1)}\le 2\left(|u'(1)|^2+\frac{\|\sqrt{a}u''\|^2_{L^2(0,1)}}{a(1)(2-K)}	\right)	\le 2\max\left\{1, \frac{1}{a(1)(2-K)}\right\}	\|u\|_{2, \circ}^2, \end{equation}
			\[
	\|u\|^2_2 \le \max\left\{2, 1+ \frac{2}{a(1)(2-K) }\right\}\|u\|^2_{2,\circ}		
			\]
			and
			\[
			\|u\|^2_{2,a} \le \max\left\{4, \frac{4}{a(1)(2-K)}+ 1\right\} \|u\|_{2, \circ}^2
			\]
	\end{Proposition}
	\begin{proof} 
	Fix $u\in \mathcal \mathcal K^2_a(0,1)$ and take $x \in (0,1]$. 
		Since $u(0)=0$, we have that \eqref{star} and \eqref{star1} still hold.
Moreover,
		\begin{equation*}
			|u'(1)- u'(x)|=\left|\int_x^1\frac{\sqrt{a(t)}u''(t)}{\sqrt{a(t)}}dt\right| \le 	\|\sqrt{a}u''\|_{L^2(0,1)}\Biggl (\int_x^1\frac{1}{a(t)}dt\Biggr )^{\frac{1}{2}},
		\end{equation*}
		for all $x \in (0,1)$.
		Consequently, proceeding as in \eqref{virizion} and \eqref{2.12'},
		\begin{equation}\label{virizion2}
			\begin{aligned}
				\int_0^1|u'(x)-u'(1)|^2dx&\le \|\sqrt{a}u''\|^2_{L^2(0,1)}\int_0^1dx\int_x^1\frac{1}{a(t)}dt\\
				&\le \frac{\|\sqrt{a}u''\|^2_{L^2(0,1)}}{a(1)(2-K)}.
			\end{aligned}
		\end{equation}
		Now, we estimate  $|u'(1)|^2$. 
By \eqref{virizion2}
\begin{equation}\label{stimau'(1)}
\begin{aligned}
|u'(1)|^2&= \left|\int_0^1 (u'(1))^2dx \right| \le 2\int_0^1 (u'(1)-u'(x))^2 dx + 2 \int_0^1 |u'(x)|^2dx\\
& \le2\left( \frac{\|\sqrt{a}u''\|^2_{L^2(0,1)}}{a(1)(2-K)}+ \|u'\|^2_{L^2(0,1)}\right).
\end{aligned}
\end{equation}
Hence, for all $u \in K^2_a(0,1)$, it results
\[
\begin{aligned}
\|u\|_{2, \circ}^2 &=  |u'(1)|^2+\|\sqrt{a}u''\|^2_{L^2(0,1)} \le \left(\frac{2}{a(1)(2-K)}+ 1\right)\|\sqrt{a}u''\|^2_{L^2(0,1)}+ 2 \|u'\|^2_{L^2(0,1)} \\
&\le \max\left\{ \frac{2}{a(1)(2-K)}+ 1, 2\right\}\|u\|_{2,a}^2.
\end{aligned}
\]
Furthermore, by \eqref{star1} and \eqref{virizion2}
	\[
	\begin{aligned}
\|u\|_{L^2(0,1)}^2&\le	\|u'\|_{L^2(0,1)}^2=	\int_0^1|u'(x)|^2dx \le\int
_0^1(|u'(x)-u'(1)|+|u'(1)|)^2 dx\\
&
\le 2|u'(1)|^2+2\int_0^1|u'(x)-u'(1)|^2dx\\
		&\le 2|u'(1)|^2+2\frac{\|\sqrt{a}u''\|^2_{L^2(0,1)}}{a(1)(2-K)}.
	\end{aligned}
\]
Thus, \eqref{stimaL2} holds and  for all $u \in K^2_a(0,1)$, it results
\[
\begin{aligned}
\|u\|_2^2 &=	\|u\|^2_{L^2(0,1)}+ \|\sqrt{a}u''\|^2_{L^2(0,1)}\le 2|u'(1)|^2 + \left(\frac{2}{a(1)(2-K)}+1\right) \|\sqrt{a}u''\|^2_{L^2(0,1)} \\
& \le \max\left\{2, \frac{2}{a(1)(2-K)}+ 1\right\} \|u\|_{2, \circ}^2.
\end{aligned}
\]
Hence
\[
\begin{aligned}
\|u\|^2_{2,a}&= \|u\|_2^2+ \|u'\|^2_{L^2(0,1)}\le 
4|u'(1)|^2+\left(\frac{4}{a(1)(2-K)}+1 \right)
 \|\sqrt{a}u''\|^2_{L^2(0,1)}\\
 &\le
 \max\left\{4, \frac{4}{a(1)(2-K)}+ 1\right\} \|u\|_{2, \circ}^2.
\end{aligned}
\]
	\end{proof}

To complete the description of the functional setting needed to treat problem (\ref{(P_feed)}), we consider the space
\[
\mathcal W(0,1):=\{u\in K^2_a(0,1): au''\in H^2(0,1)\},
\]
where \eqref{GG0} becomes
\begin{equation}\label{GGnew}
\int_0^1 (au'')''v\,dx = [(au'')'v](1)- [au''v']_{x=0}^{x=1} + \int_0^1 au''v''dx
\end{equation}
for all   $(u,v) \in \mathcal W(0,1) \times K^2_a(0,1)$.
Now, define
\[
K_{a,0}^2(0,1):=\{ u \in K^2_a(0,1): u'(0)=0, \text{ when $a$ is (WD)}\}
\]
and
 \[
		 \begin{aligned}
		\mathcal W_0 (0,1):=\{ u \in \mathcal W (0,1): &\; u'(0)=0 \text{ if $a$ is (WD), or}\\&
 \;(au'')(0)=0, \text{  if $a$ is (SD)}\}. 
		 \end{aligned}\]
	Observe that, if $a$ is (SD), then $K^2_{a,0}(0,1)\equiv K^2_a(0,1)$.
		Clearly,   for all $(u,v)\in \mathcal W_0 (0,1) \times K^2_{a,0}(0,1)$ \eqref{GGnew} becomes
	\begin{equation}\label{GF0new}
			\int_{0}^{1}(au'')''v\,dx=[(au'')'v](1) - [au''v'](1)+\int_{0}^{1}au''v''dx.
		\end{equation}
		Indeed,  it is sufficient to prove that, if $a$ is (SD), then there exists $\lim_{\delta \rightarrow 0^+}(au''v')(\delta)=0$. To this aim, one can write
		\[
		\begin{aligned}
		(au''v')(\delta)&=(au''v')(1)-\int_\delta^1 (au''v')'(x)dx \\
		&=(au''v')(1) -\int_\delta^1 (au'')'(x)v'(x)dx - \int_\delta^1 (au''v'')(x)dx,
	\end{aligned}	\]
		where $\delta \in (0,1)$ is fixed, and, proceeding as in Subsection \ref{section 3}.1, one has the thesis.

The last space needed to study the well posedness of \eqref{(P_feed)} is  the following one:
	\begin{equation*}
		\mathcal{K}_0:=K^2_{a,0}(0,1)\times L^2(0,1),
	\end{equation*}
	endowed with inner product and norm given by
	\[
		\langle (u,v),(\tilde{u},\tilde{v})\rangle_{\mathcal{K}_0}:=\langle (u,v),(\tilde{u},\tilde{v})\rangle_{\mathcal{H}_0}+\beta u(1)\tilde{u}(1)+\gamma u'(1)\tilde{u}'(1)
\]
 and 
\[
		\|(u,v)\|^2_{\mathcal{K}_0}:=	\|(u,v)\|^2_{\mathcal{H}_0}+\beta u^2(1)+\gamma (u')^2(1)
	\]
		for every $(u,v), (\tilde{u},\tilde{v})\in\mathcal{K}_0$.
		
		 Now, consider the operator $(B, D(B))$ defined as $By:= (ay_{xx})_{xx}$ for all 
		$
		y\in   D(B):=\mathcal W_0(0,1)$ and
define $\mathcal{B}:D(\mathcal{B})\subset\mathcal{K}_0\to \mathcal{K}_0$ as
	\[	\mathcal{B}:=\begin{pmatrix}
		0 & Id \\
		-B & 0
	\end{pmatrix},\]
	\[
	\begin{aligned}
	D(\mathcal{B}):=\{(u,v) \in D(B)\times  K^2_{a,0}(0,1): &\; \beta u(1)-(au'')'(1)+v(1)=0,\\ &\; \gamma u'(1)+(au'')(1)+v'(1)=0\}.
	\end{aligned}
	\]
\begin{Theorem}\label{th generazionenew}
	Assume $a$ (WD) or (SD). Then the operator $(\mathcal{B}, D(\mathcal{B}))$ is non positive with dense domain and generates a contraction semigroup $(R(t))_{t\ge 0}$.
\end{Theorem}
We postpone the proof of the previous result to the Appendix. Moreover, as in the Subsection \ref{subsect2}, one has the following existence theorem.
	\begin{Theorem}\label{Theorem 2.6new}
		Assume $a$ (WD) or (SD).
		If $(y^0_T,y^1_T)\in\mathcal{K}_0$, then there exists a unique mild solution
		\begin{equation*}
			y\in \mathcal{C}^1([0,+\infty);L^2(0,1))\cap \mathcal{C}([0,+\infty);K^2_{a,0}(0,1))
		\end{equation*}
		of (\ref{(P_feed)}) which depends continuously on the initial data $(y^0_T,y^1_T)\in \mathcal{K}_0$. Moreover, if $(y^0_T,y^1_T)\in D(\mathcal{B})$, then the solution $y$ is classical, in the sense that
		\begin{equation*}
			y\in \mathcal{C}^2([0,+\infty);L^2(0,1))\cap \mathcal{C}^1([0,+\infty);K^2_{a,0}(0,1))\cap \mathcal{C}([0,+\infty); D(B))
		\end{equation*}
		and the equation of (\ref{(P_feed)}) holds for all $t\ge 0$.
	\end{Theorem}
	
	Remark \ref{remark revers} still holds in this case.

	In the last result of this section we study the degenerate elliptic equation associated to \eqref{(P_feed)}, whose analysis is crucial to obtain the stabilization of problem (\ref{(P_feed)}). In particular, we prove that the corresponding elliptic problem admits a unique solution satisfying suitable estimates.
	\begin{Proposition}\label{prob variazionale}
		Assume $a$ (WD) or (SD) and consider $\beta\ge0$ and $\gamma >0$. Define
		\begin{equation*}
			|||z|||^2:=\int_0^1a(z'')^2dx+\beta z^2(1)+\gamma (z'(1))^2
		\end{equation*}
		for all $z\in K^2_{a,0}(0,1)$. Then the norms $|||\cdot |||$, $\|\cdot\|_{2,a}$, $\|\cdot\|_2$ and $||\cdot ||_{2, \circ}$ are equivalent in $K^2_{a,0}(0,1)$. Moreover, for every $\lambda,\mu \in\mathbb{R}$, the variational problem
		\begin{equation*}
			\int_0^1az''\varphi''dx+\beta z(1)\varphi(1)+\gamma z'(1)\varphi'(1)=\lambda\varphi(1)+\mu\varphi'(1) \quad \forall \; \varphi \in K^2_{a,0}(0,1)
		\end{equation*}
		admits a unique solution $z\in K^2_{a,0}(0,1)$ which satisfies the estimates
		\begin{equation}\label{eroeferreo}
			\norm{z}^2_{L^2(0, 1)}\le C_1C_2C_{a,K, \lambda, \mu}^2\,\,\,\,\,\,\text{ and }\,\,\,\,\,\,	|||z|||^2\le C^2_{a,K, \lambda, \mu},
		\end{equation}
		where 
		\begin{equation}\label{C1}
C_1:=	 2\max\left\{1, \frac{1}{a(1)(2-K)}\right\}, \quad C_2 := \max\left\{1, \frac{1}{\gamma}\right\}\end{equation}and \[C_{a,K, \lambda, \mu}:= \sqrt{C_2}\left(|\lambda|\sqrt{C_1} +|\mu|\right).
\]
		In addition $z\in D(B)$ and solves
		\begin{equation}\label{falenaferrea}
			\begin{cases}
				Bz=0, \\
				\beta z(1)-(az'')'(1)=\lambda,\\
				\gamma z'(1)+(az'')(1)=\mu.
			\end{cases}
		\end{equation}
	\end{Proposition}
\begin{proof} 
Clearly, 
for all $z \in K^2_{a,0}(0,1)$,
$\ds\|z\|^2_{2, \circ}\le \max\left\{ 1, \frac{1}{\gamma}\right\}|||z|||^2$. Moreover, by \eqref{star}, 
\[
|||z|||^2 \le \max\{1, \gamma\}\|z\|_{2, \circ}^2+\beta  \|z\|^2_{2,a},
\]
 for all $z \in K^2_{a,0}(0,1)$.
The first part of the thesis follows by Proposition \ref{normeequivalenti}, which holds in $K^2_{a}(0,1)$ and, in particular, in $K^2_{a,0}(0,1)$.

Now, consider the bilinear and symmetric form $\Lambda: K^2_{a,0}(0,1) \times K^2_{a,0}(0,1) \rightarrow \R$ such that
\[
\Lambda (z, \varphi) := \int_0^1 az'' \varphi'' dx + \beta z(1)\varphi(1)+\gamma z'(1)\varphi'(1).
\]
Clearly, $\Lambda$ is coercive and continuous. Now, consider the linear functional
\[
\Gamma( \varphi):= \lambda \varphi(1)+\mu \varphi'(1),
\]
with $\varphi \in K^2_{a,0}(0,1)$ and $\lambda,\mu\in\mathbb{R}$. The functional $ \Gamma$ is continuous and linear; thus, by the Lax-Milgram Theorem, there exists a unique solution $z \in K^2_{a,0}(0,1)$ of
\begin{equation}\label{05}
\Lambda (z, \varphi)= \Gamma (\varphi)
\end{equation}
for all $\varphi \in K^2_{a,0}(0,1)$; in particular,
\begin{equation}\label{04}
\int_0^1 a(z'')^2dx+ \beta z^2(1)+\gamma (z'(1))^2 =  \ \lambda z(1)+\mu z'(1).
\end{equation}
Now, we will prove \eqref{eroeferreo}. To this aim observe that
\begin{equation}\label{termineux}
|z(x)|^2 \le \|z'\|^2_{L^2(0,1)}\le 2\max\left\{ 1, \frac{1}{a(1)(2-K)}\right\}\|z\|^2_{2, \circ}
\end{equation}
for all $x \in (0,1]$,
and, by the definition of $\|\cdot\|^2_{2, \circ}$,
\begin{equation}\label{stimaz'(1)}
|z'(1)|^2 \le \|z\|^2_{2, \circ}
\end{equation}
 for all $z \in K^2_{a,0}(0,1)$.
By \eqref{04}, \eqref{termineux} and \eqref{stimaz'(1)} we have 
\[
\begin{aligned}
|||z|||^2 &= \lambda z(1)+\mu z'(1) \le 
|\lambda|\sqrt{2\max\left\{ 1, \frac{1}{a(1)(2-K)}\right\}}\|z\|_{2, \circ}+ |\mu| \|z\|_{2, \circ}\\
& \le C_{a,K, \lambda, \mu}|||z|||,
\end{aligned}
\]
thus
\[
|||z|||\le C_{a,K, \lambda, \mu}.
\]
Moreover, by \eqref{stimaL2},  we have

\[
\begin{aligned}
\|z\|^2_{L^2(0,1)} &\le 2\max\left\{1, \frac{1}{a(1)(2-K)}\right\}	\|z\|_{2, \circ}^2 \le   2\max\left\{1, \frac{1}{a(1)(2-K)}\right\}\max\left\{ 1, \frac{1}{\gamma}\right\}|||z|||^2
\\
&\le  2\max\left\{1, \frac{1}{a(1)(2-K)}\right\}\max\left\{ 1, \frac{1}{\gamma}\right\} C_{a,K, \lambda, \mu}^2.
\end{aligned}
\]

Now, we will prove that $z$ belongs to $D(B)$ and solves \eqref{falenaferrea}. To this aim,  consider again \eqref{05}; clearly, it holds for every $\varphi \in C_c^\infty(0,1)$, so that 
$
\int_0^1 a z''\varphi''dx=0 \mbox{ for all }\varphi \in C_c^\infty(0,1).
$
Thus
$az''$ is constant a.e. in $(0,1)$ (see, e.g., \cite[Lemma 1.2.1]{JLJ}) and so 
$
(az'')'' =0$ a.e.  in  $(0,1),
$
in particular
$Az= (az'')'' \in L^2(0,1)$ (this implies $az'' \in H^2(0,1)$, by \cite[Lemma 2.1]{CF}) and $z'(0)=0$ in the weakly degenerate case, since $z \in K^2_{a,0}(0,1)$.

Now, coming back to \eqref{05} and using \eqref{GGnew} or \eqref{GF0new}, for all $\varphi \in K^2_{a,0}(0,1)$ we have
\begin{equation*}
	\begin{aligned}
		&\int_0^1 az''\varphi'' dx+ \beta z(1)\varphi(1)+\gamma z'(1)\varphi'(1)=\lambda \varphi(1)+\mu \varphi'(1)\\ &\Longleftrightarrow -(az'')'(1)\varphi(1)+(az''\varphi')(1)+\beta z(1)\varphi(1)+\gamma (z'\varphi')(1) = \lambda \varphi(1)+\mu \varphi'(1)
	\end{aligned}
\end{equation*}
in the weakly degenerate case and 
\begin{equation*}
	\begin{aligned}
		&\int_0^1 az''\varphi'' dx+ \beta z(1)\varphi(1)+\gamma z'(1)\varphi'(1)=\lambda \varphi(1)+\mu \varphi'(1)\\ &\Longleftrightarrow -(az'')'(1)\varphi(1)+(az''\varphi')(1) - (az''\varphi')(0)+\beta z(1)\varphi(1)+\gamma (z'\varphi')(1) = \lambda \varphi(1)+\mu \varphi'(1),
	\end{aligned}
\end{equation*}
in the strongly one. Thus,
$
-(az'')'(1) + \beta z(1)=\lambda$ and $\gamma z'(1)+(az'')(1)=\mu,
$ and, in the strongly degenerate case $(az'')(0)=0$;
that is  $z \in  D(B)$ and solves \eqref{falenaferrea}.
\end{proof}
		\subsection{Stabilization in the case $\beta,\gamma > 0$}\label{section 6}
	In this subsection we will study the stability of \eqref{(P_feed)} in the case $\beta,\gamma > 0$. To this aim we give the following definition. 	
	\begin{Definition}
		Let $y$ be a mild solution of (\ref{(P_feed)}) and define its energy  as
		\begin{equation*}
			E_y(t):=\frac{1}{2}\int_0^1 \Biggl (y^2_t(t,x)+a(x)y^2_{xx}(t,x) \Biggr )dx+\frac{\beta}{2}y^2(t,1)+\frac{\gamma}{2}y_x^2(t,1)\quad\,\,\,\,\,\,\forall\;t\ge 0.
		\end{equation*}
	\end{Definition}
Thus, if $y$ is a mild solution, then
	\begin{equation}\label{stimapun}
		y^2(t,1) \le \frac{2}{\beta}E_y(t),
	\end{equation}
	\begin{equation}\label{stimapun0}
		y_x^2(t,1) \le \frac{2}{\gamma}E_y(t).
	\end{equation}

	In particular, it is possible to prove that the energy is a non increasing function.
	
	\begin{Theorem}\label{teorema energia decr}
		Assume $a$ (WD) or (SD) and let $y$ be a  classical solution of (\ref{(P)}). Then the energy is non increasing. In particular,
		\[
			\frac{dE_y(t)}{dt}=-y^2_t(t,1)-y^2_{tx}(t,1),\quad\,\,\,\,\,\,\,\forall\;t\ge 0.
		\]
	\end{Theorem}
	\begin{proof}
		Multiplying the equation $y_{tt}+Ay=0$ by $y_t$ and integrating over $(0,1)$, we have
		\[
		0=\frac{1}{2}\int_0^1 \Bigl (y^2_t\Bigr )_tdx+\int_0^1 (ay_{xx})_{xx}\,y_tdx.
		\]
		Using the boundary conditions and \eqref{GF0new}, one has 
		\begin{equation*}
			\begin{aligned}
				0&=\frac{1}{2}\int_0^1 \Bigl (y^2_t\Bigr )_tdx+(ay_{xx})_x(t,1)y_t(t,1)-y_{tx}(t,1)a(1)y_{xx}(t,1) +\int_0^1ay_{xx}y_{txx}dx\\
				&=\frac{1}{2}\frac{d}{dt}\Biggl (\int_0^1\Bigl (y^2_t+ay^2_{xx}\Bigr )dx+\beta y^2(t,1)+\gamma y_x^2(t,1)\Biggr )+y^2_t(t,1)+y_{tx}^2(t,1)\\ &=\frac{d}{dt}E_y(t)+y^2_t(t,1)+y_{tx}^2(t,1).
			\end{aligned}
		\end{equation*}
		Hence
		\begin{equation*}
			\frac{d}{dt}E_y(t)=-y^2_t(t,1)-y^2_{tx}(t,1)\le 0,\,\,\,\,\,\,\,\quad\forall\; t\ge 0
		\end{equation*}
		and, consequently, the energy $E_y$ associated to $y$ is non increasing.
	\end{proof}

	\begin{Proposition}\label{Prop 3.1}
		Assume  $a$  (WD)  or (SD). If $y$ is a classical solution of (\ref{(P_feed)}), then 
		\begin{equation}\label{prima uguaglianza}
			\begin{aligned}
				0&=2\int_0^1\Bigl [xy_ty_x \Bigr ]^{t=T}_{t=s}dx-\int_s^Ty^2_t(t,1)dt+\int_{Q_s}y^2_tdx\,dt\\
				&+\int_{Q_s} (3a-xa')y_{xx}^2dx\,dt+2\beta\int_s^Ty_x(t,1)y(t,1)dt+2\int_s^Ty_x(t,1)y_t(t,1)dt\\
				&+2\gamma\int_s^Ty_x^2(t,1)dt+2\int_s^Ty_x(t,1)y_{tx}(t,1)dt  -  \int_s^T a(1)y_{xx}^2(t,1)dt
			\end{aligned}
		\end{equation}
		for every $0<s<T$. Here $Q_s:=(s,T)\times (0,1)$.
	\end{Proposition}
	\begin{proof}
		Fix $s\in (0,T)$. Multiplying the equation in (\ref{(P_feed)}) by $xy_x$ and integrating over $Q_s$, we have 
		\begin{equation}\label{int parti}
			0=\int_{Q_s}y_{tt}xy_xdx\,dt+\int_{Q_s}(ay_{xx})_{xx}xy_xdx\,dt.
		\end{equation}
		Clearly the previous integrals are well defined and
		integrating by parts the first integral in (\ref{int parti}), we obtain
		\begin{equation}\label{intparti1}
			\begin{aligned}
				\int_{Q_s}y_{tt}xy_xdx\,dt=\int_0^1\Bigl [y_txy_x \Bigr ]^{t=T}_{t=s}dx-\frac{1}{2}\int_s^Ty_t^2(t,1)dt+\frac{1}{2}\int_{Q_s}y^2_tdx\,dt,
			\end{aligned}
		\end{equation}
		since $y_t(t,0)=0$.
		Now, fix $\delta \in (0, 1)$; then, 
		\begin{equation}\label{intparti3}
			\begin{aligned}
				\int_s^T\int_\delta^1 (ay_{xx})_{xx}xy_xdx\,dt&= \int_s^T [xy_x(ay_{xx})_x]_{x=\delta}^{x=1}dt - \int_s^T [(xy_x)_xay_{xx}]_{x=\delta}^{x=1}dt \\
				&+ \int_s^T\int_\delta^1 a(xy_x)_{xx}y_{xx}dx\,dt\\
						&=\int_s^T (y_x(ay_{xx})_x)(t,1)dt -  \int_s^T (xy_x(ay_{xx})_x)(t, \delta)dt\\
						&-\int_s^T(y_xay_{xx})(t,1)dt-  \int_s^T a(1)y_{xx}^2(t,1)dt \\
						&+ \int_s^T (y_xay_{xx})(t, \delta) dt + \int_s^T(xay_{xx}^2)(t, \delta)dt\\
				& + \int_s^T\int_\delta^1 a(xy_x)_{xx}y_{xx}dx\,dt.
			\end{aligned}
		\end{equation}
		Now, consider  the term
		$\int_\delta^1 a(xy_x)_{xx}y_{xx}dx$.
		Clearly,
		\[
		\begin{aligned}
			\int_\delta^1 a(xy_x)_{xx}y_{xx}dx &= 2\int_\delta^1 ay_{xx}^2dx + \frac{1}{2}\int_\delta^1 ax(y_{xx}^2)_xdx\\
			&=2\int_\delta^1 ay_{xx}^2dx  + \frac{1}{2} (ay_{xx}^2)(t,1) - \frac{1}{2} (xay_{xx}^2)(t,\delta)  -  \frac{1}{2}\int_\delta^1 ay_{xx}^2 dx\\
			&-\frac{1}{2}\int_\delta^1xa'y^2_{xx}dx.
		\end{aligned}
		\]
		By Lemmas \ref{boundaryconditions_WD} and \ref{boundaryconditions_SD}, it follows
		\[
		\lim_{\delta\to 0} (xay_{xx}^2)(t,\delta) =	\lim_{\delta\to 0}(xy_x(ay_{xx})_x)(t, \delta)=\lim_{\delta\to 0} (y_xay_{xx})(t, \delta) =0;
		\]
		moreover,  $xa'y^2_{xx}\in L^1(0,1)$. Hence, by the absolute continuity of the integral, one has
		\[
		\lim_{\delta\to 0}\int_\delta^1 a(xy_x)_{xx}y_{xx}dx = \frac{3}{2}\int_0^1 ay_{xx}^2dx  + \frac{1}{2}a(1) y_{xx}^2(t,1)-\frac{1}{2}\int_0^1xa'y^2_{xx}dx.
		\]
		Thus, by \eqref{intparti3}, one has the following equality
		\begin{equation}\label{intparti4}
			\begin{aligned}
				\int_{Q_s}(ay_{xx})_{xx} xy_xdx\,dt &=\int_s^T (y_x(ay_{xx})_x)(t,1)dt-\int_s^T(y_xay_{xx})(t,1)dt\\ &+\frac{3}{2}\int_{Q_s} ay_{xx}^2dx\,dt   -\frac{1}{2}   \int_s^T a(1)y_{xx}^2(t,1)dt-\frac{1}{2}\int_{Q_s}xa'y^2_{xx}dx\,dt.
			\end{aligned}
		\end{equation}
		Furthermore, by \eqref{int parti}, \eqref{intparti1} and \eqref{intparti4}, one has
		\[
			\begin{aligned}
				0&=\int_0^1\Bigl [y_txy_x \Bigr ]^{t=T}_{t=s}dx-\frac{1}{2}\int_s^Ty_t^2(t,1)dt+\frac{1}{2}\int_{Q_s}y^2_tdx\,dt\\
				&+\frac{3}{2}\int_{Q_s} ay_{xx}^2dx\,dt+\beta\int_s^Ty_x(t,1)y(t,1)dt+\int_s^Ty_x(t,1)y_t(t,1)dt\\
				&+\gamma\int_s^Ty_x^2(t,1)dt+\int_s^Ty_x(t,1)y_{tx}(t,1)dt  -\frac{1}{2}   \int_s^T a(1)y_{xx}^2(t,1)dt-\frac{1}{2}\int_{Q_s}xa'y^2_{xx}dx\,dt.
			\end{aligned}
		\]
		Multiplying the previous equality by $2$ we have the thesis.
	\end{proof}
	
	As a consequence of the previous equality, we have the next relation. 
	\begin{Proposition}\label{Prop 4.4}
		Assume  $a$  (WD)  or (SD). If $y$ is a classical solution of \eqref{(P_feed)}, then for all $0<s<T$ we have 
		\begin{equation}\label{equazione 2}
			\int_{Q_s}y^2_t\Bigl (\frac{K}{2}+1\Bigr )dx\,dt+\int_{Q_s}y^2_{xx}\Bigl (3a-xa'-\frac{K}{2}a\Bigr )dx\,dt
			=(B.T.),
		\end{equation}
		where 
		\[
		\begin{aligned}
			(B.T.)&= \frac{K}{2}\int_0^1\Bigl [yy_t \Bigr ]^{t=T}_{t=s}dx-2\int_0^1\Bigl [xy_ty_x \Bigr ]^{t=T}_{t=s}dx+\frac{K\beta}{2}\int_s^Ty^2(t,1)dt\\
			&+\frac{K}{2}\int_s^Ty(t,1)y_t(t,1)dt+\gamma\Bigl (\frac{K}{2}-2\Bigr )\int_s^Ty_x^2(t,1)dt\\ &+\Bigl (\frac{K}{2}-2\Bigr )\int_s^Ty_x(t,1)y_{tx}(t,1)dt+\int_s^Ty_t^2(t,1)dt-2\beta\int_s^Ty_x(t,1)y(t,1)dt\\
			&-2\int_s^Ty_x(t,1)y_t(t,1)dt+\int_s^Ta(1)y_{xx}^2(t,1)dt.
		\end{aligned}
		\]
	\end{Proposition}
	\begin{proof}
		Let $y$ be a classical solution of (\ref{(P_feed)}) and fix $s\in (0,T)$. Multiplying the equation in (\ref{(P_feed)}) by $y$, integrating over $Q_s$ and using (\ref{GF0new}), we obtain
		\begin{equation}\label{eq sommare1}
			\begin{aligned}
				0&=\int_{Q_s}y_{tt}ydx\,dt+\int_{Q_s}y(ay_{xx})_{xx}dx\,dt\\
				&=\int_0^1\Bigl [y_ty \Bigr ]^{t=T}_{t=s}dx-\int_{Q_s}y_t^2dx\,dt+\int_s^T(y(ay_{xx})_x)(t,1)dt\\
				&-\int_s^T(y_xay_{xx})(t,1)dt+\int_{Q_s}ay^2_{xx}dx\,dt.
			\end{aligned}
		\end{equation}
		Obviously,  all the previous integrals make sense and multiplying (\ref{eq sommare1}) by $\displaystyle\frac{K}{2}$, one has
		\begin{equation}\label{eq sommare2}
			\begin{aligned}
				0&=\frac{K}{2}\int_{Q_s}\Bigl (-y_t^2+ay^2_{xx}\Bigr )dx\,dt+\frac{K}{2}\int_0^1\Bigl [y_ty \Bigr ]^{t=T}_{t=s}dx\\
				&+\frac{K}{2}\int_s^Ty(t,1)(ay_{xx})_x(t,1)dt-\frac{K}{2}\int_s^Ta(1)y_x(t,1)y_{xx}(t,1)dt.
			\end{aligned}
		\end{equation}
		By summing (\ref{prima uguaglianza}) and (\ref{eq sommare2}), we get the thesis. 
	\end{proof}
	
	By \eqref{equazione 2}, we can get the next estimate.
	\begin{Proposition}\label{Prop 3.4}
		Assume $a$ (WD) or (SD)  and let $y$ be a classical solution of (\ref{(P_feed)}). Then there exists $\varepsilon_0>0$ such that for any $0<s<T$ 
			\[
			\begin{aligned}
				\frac{\varepsilon_0}{2}\int_{Q_s}\Biggl (y^2_t+ay^2_{xx} \Biggr )dx\,dt&\le\left( \vartheta + \varrho+\varsigma+ \Biggl (2-\frac{K}{2}\Biggr )\frac{1}{\gamma}\right) E_y(s)\\
	&+ \left(\frac{K}{4}+\frac{K\beta}{2}+ \beta\right) \int_s^Ty^2(t,1)dt+ \left(\beta+ 1+ \frac{2\gamma^2}{a(1)}\right)\int_s^Ty_x^2(t,1)dt,
			\end{aligned}
		\]
		where 
		\[\ds\vartheta:=   K\max\Biggl \{1,\frac{2}{a(1)(2-K)}, \frac{2}{\gamma}\Biggr \}, \; \ds\varrho:= 4 \max\left\{ \frac{2}{\gamma}, \frac{2}{a(1)(2-K)}, 1\right\}\]
		and
		\begin{equation}\label{varsigma}
		\ds\varsigma:= \max\left\{\frac{K}{4}+ 2, \frac{2}{a(1)}\right\}.\end{equation}
	\end{Proposition}
	\begin{proof} Since by assumption $K<2$, there exists $\varepsilon _0>0$ such that $2-K \ge \varepsilon _0$. Thus,
		\begin{equation*}
			3a-xa'-\frac{K}{2}a=\frac{3(2-K)a+2(Ka-xa')}{2}\ge \frac{\varepsilon_0}{2}a\,\,\,\text{ and }\,\,\,1+\frac{K}{2}\ge \frac{2-K}{2}\ge  \frac{\varepsilon_0}{2} .
		\end{equation*}
		As a consequence, the boundary terms  in (\ref{equazione 2}) can be estimated by below in the following way
		\begin{equation}\label{saccoferreo}
			(B.T.)=\int_{Q_s}y^2_t\Bigl (1+\frac{K}{2}\Bigr )dx\,dt+\int_{Q_s}y^2_{xx}\Bigl (3a-\frac{K}{2}a-xa'\Bigr )dx\,dt\ge \frac{\varepsilon_0}{2}\int_{Q_s}\Biggl (y^2_t+ay^2_{xx} \Biggr )dx\,dt.
		\end{equation}
		Now, we estimate the boundary terms from above. By Proposition \ref{normeequivalenti}, one has
	\begin{equation*}
			\begin{aligned}
				\Biggl |\int_0^1y(\tau,x)y_t(\tau,x)dx \Biggr |&\le \frac{1}{2}\int_0^1y^2_t(\tau,x)dx+\frac{1}{2}\int_0^1y^2(\tau,x)dx\\
					&\le \frac{1}{2}\int_0^1y^2_t(\tau,x)dx+
			 y_x^2(\tau, 1)+\frac{1}{a(1)(2-K)}\int_0^1a(x)y_{xx}^2(\tau,x)dx\\
				&\le \max\Biggl \{1,\frac{2}{a(1)(2-K)}, \frac{2}{\gamma}\Biggr \}E_y(\tau)
			\end{aligned}
		\end{equation*}
		for all $\tau \in [s, T]$. Hence, by Theorem \ref{teorema energia decr},
		\begin{equation}\label{stima*}
		\int_0^1\Bigl [y(\tau,x)y_t(\tau,x)\Bigr]_{\tau=s}^{\tau=T}dx \le  2\max\Biggl \{1,\frac{2}{a(1)(2-K)}, \frac{2}{\gamma}\Biggr \}E_y(s).
		\end{equation}
	Using again Proposition \ref{normeequivalenti}, one has
		\[
		\begin{aligned}
2\int_0^1xy_x(\tau,x)y_t(\tau,x)dx
			&\le \int_0^1y_x^2(\tau,x)dx + \int_0^1 y_t^2(\tau,x) dx\\
			&\le  2|y_x(\tau, 1)|^2+\frac{2}{a(1)(2-K)}\int_0^1a(x)y_{xx}^2(\tau,x)dx + \int_0^1 y_t^2(\tau,x) dx\\
			&\le 2 \max\left\{ \frac{2}{\gamma}, \frac{2}{a(1)(2-K)}, 1\right\} E_y(\tau),
		\end{aligned}
		\]
			for all $\tau \in [s, T]$,
		and, by Theorem \ref{teorema energia decr},

	\begin{equation}\label{stima**}
		2\int_0^1\Bigl [xy_x(\tau,x)y_t(\tau,x) \Bigr ]_{\tau=s}^{\tau=T}dx \le  4 \max\left\{ \frac{2}{\gamma}, \frac{2}{a(1)(2-K)}, 1\right\} E_y(s).
		\end{equation}
		Now, since $K<2$, we have
\begin{equation}\label{primo BT}
\begin{aligned}
	&\gamma\Biggl (\frac{K}{2}-2\Biggr )\int_s^Ty_x^2(t,1)dt+\Biggl (\frac{K}{2}-2\Biggr )\int_s^Ty_x(t,1)y_{tx}(t,1)dt \\ &\le \Biggl (\frac{K}{2}-2\Biggr )\frac{1}{2}(y_x^2(T,1)-y_x^2(s,1))\le \left(2-\frac{K}{2}\right)\frac{1}{2}y_x^2(s,1)
	\\&\le \Biggl (2-\frac{K}{2}\Biggr )\frac{1}{\gamma}E_y(s).
\end{aligned}
\end{equation}
Obviously
\begin{equation}\label{secondo BT}
	\frac{K}{2}\int_s^Ty(t,1)y_t(t,1)dt\le \frac{K}{4}\int_s^Ty^2(t,1)dt+\frac{K}{4}\int_s^Ty_t^2(t,1)dt,
\end{equation}
\begin{equation}\label{terzo BT}
2\beta \int_s^Ty_x(t,1)y(t,1)dt\le \beta\int_s^Ty_x^2(t,1)dt+\beta\int_s^Ty^2(t,1)dt
\end{equation}
and
\begin{equation}\label{quarto BT}
 \int_s^T2y_x(t,1)y_t(t,1)dt\le \int_s^Ty_x^2(t,1)dt+\int_s^Ty_t^2(t,1)dt.
\end{equation}
Furthermore, recalling that $\gamma y_x(t,1)+a(1)y_{xx}(t,1)+y_{tx}(t,1)=0$,
\begin{equation}\label{quinto BT}
	\frac{1}{a(1)}\int_s^Ta^2(1)y^2_{xx}(t,1)dt\le \frac{2}{a(1)}\gamma^2\int_s^Ty_x^2(t,1)dt+\frac{2}{a(1)}\int_s^Ty_{tx}^2(t,1)dt.
\end{equation}
Hence, by \eqref{saccoferreo}, \eqref{stima*}-\eqref{quinto BT}
\[
\begin{aligned}
	&	\frac{\varepsilon_0}{2}\int_{Q_s}\Biggl (y^2_t+ay^2_{xx} \Biggr )dx\,dt\le  K\max\Biggl \{1,\frac{2}{a(1)(2-K)}, \frac{2}{\gamma}\Biggr \}E_y(s)\\
	&+4 \max\left\{ \frac{2}{\gamma}, \frac{2}{a(1)(2-K)}, 1\right\} E_y(s)+ \Biggl (2-\frac{K}{2}\Biggr )\frac{1}{\gamma}E_y(s)\\
	&+ \left(\frac{K}{4}+\frac{K\beta}{2}+ \beta\right) \int_s^Ty^2(t,1)dt+ \left(\beta+ 1+ \frac{2}{a(1)}\gamma^2\right)\int_s^Ty_x^2(t,1)dt\\
	&+ \left(\frac{K}{4}+ 2\right)\int_s^Ty_t^2(t,1)dt+ \frac{2}{a(1)}\int_s^Ty_{tx}^2(t,1)dt\\
	&\le  K\max\Biggl \{1,\frac{2}{a(1)(2-K)}, \frac{2}{\gamma}\Biggr \}E_y(s)\\
	&+4 \max\left\{ \frac{2}{\gamma}, \frac{2}{a(1)(2-K)}, 1\right\} E_y(s)+ \Biggl (2-\frac{K}{2}\Biggr )\frac{1}{\gamma}E_y(s)\\
	&+ \left(\frac{K}{4}+\frac{K\beta}{2}+ \beta\right) \int_s^Ty^2(t,1)dt+ \left(\beta+ 1+ \frac{2}{a(1)}\gamma^2\right)\int_s^Ty_x^2(t,1)dt\\
	&+ \max\left\{\frac{K}{4}+ 2, \frac{2}{a(1)}\right\}\int_s^T-\frac{d}{dt}E_y(t)dt\\
	& \le K\max\Biggl \{1,\frac{2}{a(1)(2-K)}, \frac{2}{\gamma}\Biggr \}E_y(s)\\
	&+4 \max\left\{ \frac{2}{\gamma}, \frac{2}{a(1)(2-K)}, 1\right\} E_y(s)+ \Biggl (2-\frac{K}{2}\Biggr )\frac{1}{\gamma}E_y(s)\\
	&+ \left(\frac{K}{4}+\frac{K\beta}{2}+ \beta\right) \int_s^Ty^2(t,1)dt+ \left(\beta+ 1+ \frac{2}{a(1)}\gamma^2\right)\int_s^Ty_x^2(t,1)dt\\
	&+ \max\left\{\frac{K}{4}+ 2, \frac{2}{a(1)}\right\} E_y(s).
\end{aligned}
\]
	\end{proof}
In the next proposition we will find an estimate from above for 
\[\ds\int_s^Ty^2(t,1)dt+\int_s^Ty_x^2(t,1)dt.\]

To this aim, since $\beta, \gamma >0$, we can consider the constant
\[
	C_\delta :=1-2C_2\max\{1,C_1\}\delta\Biggl (\frac{1}{\beta}+\frac{1}{\gamma}\Biggr ),
	\] 
	where $C_1$ and $C_2$ are the constants defined in Proposition \ref{prob variazionale}.
\begin{Proposition}\label{Prop 3.3}
	Assume $a$ (WD)  or (SD). If $y$ is a classical solution of (\ref{(P_feed)}), then for every $0<s<T$ and for every  $\delta \in (0, \nu)$
	 we have
		\[
		\begin{aligned}
			\int_s^T y^2(t,1)dt+\int_s^Ty_x^2(t,1)dt&\le\frac{2 \delta}{C_\delta}\int_s^TE_y(t)dt\\
			&+\frac{1}{C_\delta}\left(2+4 \frac{C_1^2C_2^2}{\beta} +  4\frac{C_1C_2^2}{\gamma} + \frac{1}{\delta} + 2\frac{C_1C_2^2\max\{1,C_1\}}{\delta}\right)E_y(s),
		\end{aligned}
	\]
	where
	\[
	\nu:= \frac{\beta \gamma}{2 C_2\max\{1,C_1\} (\beta+\gamma)}.
	\]
	
\end{Proposition}
\begin{proof}
Set $\lambda =y(t,1)$, $\mu =y_x(t,1)$, where $t \in [s,T]$, and let $z=z(t,\cdot)\in K^2_{a,0}(0,1)$ be the unique solution of
\begin{equation*}
	\int_0^1az_{xx}\varphi''dx+\beta z(t,1)\varphi(1)+\gamma z_x(t,1)\varphi'(1)=\lambda\varphi(1)+\mu\varphi'(1),\,\,\,\,\,\,\,\quad\forall\;\varphi\in K^2_{a,0}(0,1).
\end{equation*}
By Proposition \ref{prob variazionale}, $z(t,\cdot)\in D(B)$ for all $t$ and solves
 \begin{equation}\label{fogliaferrea}
 	\begin{cases}
				Bz=0, \\
				\beta z(t,1)-(az_{xx})_x(t,1)=\lambda,\\
				\gamma z_x(t,1)+(az_{xx})(t,1)=\mu.
			\end{cases}
 \end{equation}
Now, multiplying the equation in (\ref{(P_feed)}) by $\displaystyle z$ and integrating over $Q_s$, we have
\begin{equation}\label{acquecrespe}
\begin{aligned}
	\int_0^1\Bigl [y_tz\Bigr ]^{t=T}_{t=s}dx-\int_{Q_s}y_tz_tdx\,dt&=-\int_s^Tz(t,1)(ay_{xx})_x(t,1)dt+\int_s^Tz_x(t,1)(ay_{xx})(t,1)dt\\
		&-\int_{Q_s}az_{xx}y_{xx}dx\,dt.
\end{aligned}
\end{equation}
On the other hand, multiplying the equation in (\ref{fogliaferrea}) by $y$ and integrating over $Q_s$, we have
$
	\int_{Q_s}(az_{xx})_{xx}y\,dx\,dt=0.
$
Thus, by (\ref{GGnew}), we get
\begin{equation}\label{acquecrespenew}
\int_{Q_s}az_{xx}y_{xx}dx\,dt=-\int_s^T(az_{xx})_x(t,1)y(t,1)dt+\int_s^Ty_x(t,1)(az_{xx})(t,1)dt.
\end{equation}
Substituting \eqref{acquecrespenew} in (\ref{acquecrespe}), using the fact that $(az_{xx})_x(t,1)=\beta z(t,1)-\lambda$, $(az_{xx})(t,1)=-\gamma z_x(t,1)+\mu$, $\lambda =y(t,1)$ and $\mu=y_x(t,1)$, we have
\begin{equation*}
\begin{aligned}
&		\int_0^1\Bigl [y_tz \Bigr ]^{t=T}_{t=s}dx-\int_{Q_s}y_tz_tdx\,dt=-\int_s^Tz(t,1)(ay_{xx})_x(t,1)dt+\int_s^Tz_x(t,1)(ay_{xx})(t,1)dt\\
		&+\int_s^T(az_{xx})_x(t,1)y(t,1)dt-\int_s^Ty_x(t,1)(az_{xx})(t,1)dt\\
		&=-\int_s^Tz(t,1)(ay_{xx})_x(t,1)dt+\int_s^Tz_x(t,1)(ay_{xx})(t,1)dt+\int_s^Ty(t,1)[\beta z(t,1)-\lambda ]dt\\
		&-\int_s^Ty_x(t,1)[-\gamma z_x(t,1)+\mu ]dt\\
		&=\int_s^Tz(t,1)[\beta y(t,1)-(ay_{xx})_x(t,1)]dt\\ &+\int_s^Tz_x(t,1)[(ay_{xx})(t,1)+\gamma y_x(t,1)]dt-\int_s^Ty^2(t,1)dt-\int_s^Ty_x^2(t,1)dt.
\end{aligned}
\end{equation*}
Then 
\begin{equation}\label{solcoferreo}
\begin{aligned}
	\int_s^Ty^2(t,1)dt+\int_s^Ty_x^2(t,1)dt&=-\int_s^T(y_tz)(t,1)dt-\int_s^T(z_xy_{tx})(t,1)dt\\ &-\int_0^1 [y_tz ]^{t=T}_{t=s}dx+\int_{Q_s}y_tz_tdx\,dt.
\end{aligned}
\end{equation}
Now, we have to estimate the four terms in the previous equality.  
As a first step observe that by \eqref{eroeferreo} and the definition of $C_{a,K, \lambda, \mu}$
\begin{equation}\label{Stimaz}
\begin{aligned}
\int_0^1z^2(t,x) dx&\le C_1C_2 C_{a,K, \lambda, \mu}^2\\
&\le 2C_1^2C_2^2y^2(t,1)+2C_1C_2^2y_x^2(t,1),
\end{aligned}
\end{equation}
since $\lambda =y(t,1)$ and $\mu=y_x(t,1)$. Thus, for all $t\in [s,T]$,
\begin{equation*}
	\begin{aligned}
		\int_0^1|y_tz |(t,x)dx&\le  \frac{1}{2}\int_0^1y^2_t(t,x)dx+\frac{1}{2}\int_0^1z^2(t,x) dx\\
		&\le E_y(t)+C_1^2C_2^2y^2(t,1)+C_1C_2^2y_x^2(t,1)\\
		& \le E_y(t) +2\frac{C_1^2C_2^2}{\beta} E_y(t) +2 \frac{C_1C_2^2}{\gamma} E_y(t)\\
& \le \left(1+2\frac{C_1^2C_2^2}{\beta}   +2\frac{C_1C_2^2}{\gamma} \right)E_y(t).
	\end{aligned}
\end{equation*}
By Theorem \ref{teorema energia decr}, 
\begin{equation}\label{latios}
\int_0^1 [y_tz]^{t=T}_{t=s} dx\le 2\left(1+  2\frac{C_1^2C_2^2}{\beta}   + 2\frac{C_1C_2^2}{\gamma} \right)E_y(s).
\end{equation}
Moreover, for any $\delta >0$ we have
\begin{equation}\label{colloferreo}
	\int_s^T|(y_tz)(t,1)|dt\le \frac{1}{\delta}\int_s^Ty^2_t(t,1)dt+\delta\int_s^Tz^2(t,1)dt.
\end{equation}
By definition of $|||\cdot|||$,  one has

\begin{equation}\label{stimaz1}
\begin{aligned}
	z^2(t,1)&\le \frac{1}{\beta}|||z|||^2\le \frac{1}{\beta}C_{a,K, \lambda, \mu}^2 \le  2\frac{C_1C_2}{\beta}y^2(t,1)+  2\frac{C_2}{\beta}y_x^2(t,1)\\
	&\le  2\frac{C_2\max\{1,C_1\}}{\beta}\left(y^2(t,1)+y_x^2(t,1) \right).
	\end{aligned}
\end{equation}
Thus, by (\ref{colloferreo}), we have
\begin{equation}\label{latias}
		\int_s^T|(y_tz)(t,1)|dt\le \ds \frac{1}{\delta}\int_s^Ty^2_t(t,1)dt+2\frac{C_2\max\{1,C_1\}}{\beta}\delta\int_s^T(y^2+y_x^2)(t,1)dt.
\end{equation}
In a similar way,  it is possible to find the next estimate
\begin{equation}\label{LATIOS}
	\int_s^T|(z_xy_{tx})(t,1)|dt\le  \ds \frac{1}{\delta}\int_s^Ty^2_{tx}(t,1)dt+ 2\frac{C_2\max\{1,C_1\}}{\gamma}\delta\int_s^T(y^2+y_x^2)(t,1)dt,
\end{equation}
being
\begin{equation}\label{stimaz2}
	z_x^2(t,1)\le \frac{1}{\gamma}|||z|||^2\le  \frac{1}{\gamma}C_{a,K, \lambda, \mu}^2.
\end{equation}
Therefore, summing \eqref{latias} and \eqref{LATIOS} and applying Theorem \ref{teorema energia decr} we obtain
\begin{equation}\label{somma}
\begin{aligned}
&\int_s^T|(y_tz)(t,1)|dt+	\int_s^T|(z_xy_{tx})(t,1)|dt \\
&\le \frac{1}{\delta}\int_s^T -\frac{d}{dt}E_y(t)dt+2\max\{1,C_1\}C_2\delta\Biggl (\frac{1}{\beta}+\frac{1}{\gamma}\Biggr )\int_s^T(y^2+y_x^2)(t,1)dt\\
&\le  \frac{E_y(s)}{\delta}+2\max\{1,C_1\}C_2\delta\Biggl (\frac{1}{\beta}+\frac{1}{\gamma}\Biggr )\int_s^T(y^2+y_x^2)(t,1)dt.
\end{aligned}
\end{equation}
Finally, we estimate the last integral in (\ref{solcoferreo}), i.e. $\ds\int_{Q_s}|y_tz_t|dx\,dt$. To this aim, consider again problem (\ref{falenaferrea}) and differentiate with respect to $t$. Thus
\begin{equation*}
	\begin{cases}
		(a(z_t)_{xx})_{xx}=0, \\
		\beta z_t(t,1)-(a(z_t)_{xx})_x(t,1)=y_t(t,1), \\
		\gamma (z_t)_x(t,1)+(a(z_t)_{xx})(t,1)=(y_x)_t(t,1).
	\end{cases}
\end{equation*}
Clearly, $z_t$ satisfies \eqref{eroeferreo} and, proceeding as in \eqref{Stimaz}, we have
\[
			\norm{z_t}^2_{L^2(0, 1)}\le 2C_1^2C_2^2y_t^2(t,1)+2C_1C_2^2y_{tx}^2(t,1)\le 2C_1C_2^2 \max\{1, C_1\} (y_t^2(t,1)+y_{tx}^2(t,1)).
	\]
Thus, for $\delta >0$ we find
\begin{equation}\label{keldeo}
	\begin{aligned}
		\int_{Q_s}|y_tz_t |dx\,dt&\le \delta\int_{Q_s}y^2_tdx\,dt+\frac{1}{\delta}\int_{Q_s}z^2_tdx\,dt \\
		&\le 2\delta\int_s^TE_y(t)dt+\frac{2C_1C_2^2 \max\{1, C_1\} }{\delta}\int_s^T(y^2_t(t,1)+ y^2_{tx}(t,1))dt\\
		&=2\delta\int_s^TE_y(t)dt+\frac{2C_1C_2^2 \max\{1, C_1\} }{\delta}\int_s^T-\frac{d}{dt}E_y(t)dt\\
		&\le 2\delta\int_s^TE_y(t)dt+\frac{2C_1C_2^2 \max\{1, C_1\} }{\delta}E_y(s).
	\end{aligned}
\end{equation}
Coming back to (\ref{solcoferreo}) and using (\ref{latios}), (\ref{somma}) and (\ref{keldeo}), we obtain the thesis for every  $\delta\in (0, \nu)$.
\end{proof}

As a consequence of Propositions \ref{Prop 3.4} and \ref{Prop 3.3}, we can formulate the main result of this section, whose proof is based on  \cite[Theorem 8.1]{Ko}.
\begin{Theorem}\label{teoremaprincipale}
	Assume $a$ (WD) or (SD) and let $y$ be a mild solution of (\ref{(P)}). Then, for all $t>0$ and for all $
	\ds	\delta \in \left(0,\min \left\{\nu, \frac{\varepsilon_0}{C_3}\right\}\right),$  one has
	\[
		E_y(t)\le E_y(0)e^{1-\frac{t}{M}},
	\]
where
$
M:=\ds\frac{C_4}{\varepsilon_0 - \delta C_3},
$ 

	\[C_3:=\Biggl (\frac{K}{4}+\frac{K\beta}{2}+\beta+\varepsilon_0\frac{\beta}{2}\Biggr )\frac{2}{C_\delta} + \left(\beta +1+\frac{2\gamma^2}{a(1)}+\varepsilon_0\frac{\gamma}{2}\right)\frac{2}{C_\delta}, \]

\[
	C_4:=\vartheta + \varrho+\varsigma+ \Biggl (2-\frac{K}{2}\Biggr )\frac{1}{\gamma} + C_5
	\]
	and
\[\begin{aligned}
C_5:=&\ds\frac{1}{C_\delta}\left( 2+4 \frac{C_1^2C_2^2}{\beta} + 4\frac{C_1C_2^2}{\gamma} + \frac{1}{\delta} + \frac{2C_1C_2^2\max\{1,C_1\}}{\delta}\right) \cdot\\
&\cdot\left(\frac{K}{4}+\frac{K\beta}{2}+2\beta+\varepsilon_0\frac{\beta}{2} +1+\frac{2\gamma^2}{a(1)}+\varepsilon_0\frac{\gamma}{2}\right).
\end{aligned}\]
\end{Theorem}

	\begin{proof}
		As a first step, consider $y$ a classical solution of \eqref{(P)} and take $\delta>0$ such that
$
	\ds	\delta <\min \left\{\frac{\varepsilon_0}{C_3}, \nu\right\}.
$
	By definition of $E_y$ and Propositions \ref{Prop 3.4}, \ref{Prop 3.3}, we have
		\[
		\begin{aligned}
&		\varepsilon_0  \int_s^TE_y(t)dt=\frac{\varepsilon_0}{2}\int_{Q_s} \Biggl (y^2_t(t,x)+a(x)y^2_{xx}(t,x) \Biggr )dxdt\\
		&+\varepsilon_0\frac{\beta}{2}\int_s^Ty^2(t,1)dt+\varepsilon_0\frac{\gamma}{2}\int_s^Ty_x^2(t,1)dt\\
		&\le \left( \vartheta + \varrho+\varsigma+ \Biggl (2-\frac{K}{2}\Biggr )\frac{1}{\gamma}\right) E_y(s)+\Biggl (\frac{K}{4}+\frac{K\beta}{2}+\beta+\varepsilon_0\frac{\beta}{2}\Biggr )\int_s^Ty^2(t,1)dt\\
		&+\left(\beta +1+\frac{2\gamma^2}{a(1)}+\varepsilon_0\frac{\gamma}{2}\right)\int_s^Ty_x^2(t,1)dt\\
	&	\le  \left( \vartheta + \varrho+\varsigma+ \Biggl (2-\frac{K}{2}\Biggr )\frac{1}{\gamma}\right) E_y(s)+\Biggl (\frac{K}{4}+\frac{K\beta}{2}+\beta+\varepsilon_0\frac{\beta}{2}\Biggr )\frac{2}{C_\delta} \delta\int_s^TE_y(t)dt\\
	&+\Biggl (\frac{K}{4}+\frac{K\beta}{2}+\beta+\varepsilon_0\frac{\beta}{2}\Biggr )\frac{1}{C_\delta}\left(2+4 \frac{C_1^2C_2^2}{\beta} +  4\frac{C_1C_2^2}{\gamma} + \frac{1}{\delta} + \frac{2C_1C_2^2\max\{1,C_1\}}{\delta}\right)E_y(s)\\
	&+\left(\beta +1+\frac{2\gamma^2}{a(1)}+\varepsilon_0\frac{\gamma}{2}\right)\frac{2}{C_\delta} \delta\int_s^TE_y(t)dt\\
	&+ \left(\beta +1+\frac{2\gamma^2}{a(1)}+\varepsilon_0\frac{\gamma}{2}\right)\frac{1}{C_\delta}\left(2+4 \frac{C_1^2C_2^2}{\beta} +  4\frac{C_1C_2^2}{\gamma} + \frac{1}{\delta} + \frac{2C_1C_2^2\max\{1,C_1\}}{\delta}\right)E_y(s).
		\end{aligned}
	\]
This implies $
			\Biggl [\varepsilon_0-\delta C_3 \Biggr ]\int_s^TE_y(t)dt\le C_4E_y(s).
		$
Hence, we can apply \cite[Theorem 8.1]{Ko} with $M:= \frac{C_4}{\varepsilon_0-\delta C_3}$ and \eqref{Stabilità} holds.
	If $y$ is the mild solution of the problem, we can proceed as in \cite{Stability_Genni_Dimitri}, obtaining the thesis.
	\end{proof}
	
	\subsection{Stabilization if $\beta, \gamma \ge0$ in the weakly degenerate case}
	Through this subsection we will assume that $a$ is (WD). In this situation, we know that if $y \in K^2_{a,0}(0,1)$ then $y'(0)=0$; thus, using this property we can prove the previous results if $\beta=0$ and/or $\gamma=0$. Moreover, the same results are improved when  $\beta$ and $\gamma$ are strictly positive. Indeed
\begin{equation}\label{stimawd}
|y'(x)|= \left|\int_0^x y''(t)dt\right| \le \int_0^1\frac{\sqrt{a(t)}|y''(t)|	}{\sqrt{a(t)}}dt\le \left\| \frac{1}{a}\right\|^{\frac{1}{2}}_{L^1(0,1)} \|\sqrt{a}y''\|_{L^2(0,1)}
\end{equation}
for all $x \in (0,1]$.
In particular,
\[
|y'(1)|^2 \le  \left\| \frac{1}{a}\right\|_{L^1(0,1)} \|\sqrt{a}y''\|_{L^2(0,1)}^2.
\]
Hence, for all $y \in K^2_{a,0}(0,1)$,
\[
\begin{aligned}
\|y\|^2_{2,\circ}&= |y'(1)|^2 +  \|\sqrt{a} y''\|_{L^2(0,1)}^2\le \left(1+   \left\| \frac{1}{a}\right\|_{L^1(0,1)}\right) \|\sqrt{a} y''\|_{L^2(0,1)}^2\\
&\le  \left(1+   \left\| \frac{1}{a}\right\|_{L^1(0,1)}\right) |||y|||^2.
\end{aligned}
\]
Observe that the previous inequality holds also if $\beta$ and $\gamma$ are non negative.

Hence, if $a$ is (WD), Proposition \ref{prob variazionale} still holds with the following estimates
		\begin{equation}\label{eroeferreonew}
			\norm{y}^2_{L^2(0, 1)}\le2\max\left\{1, \frac{1}{a(1)(2-K)}\right\}A_\gamma C_{a,K, \lambda, \mu}^2\,\,\,\,\,\,\text{ and }\,\,\,\,\,\,	|||y|||^2\le C^2_{a,K, \lambda, \mu},
		\end{equation}
		where 
		\[
		A_\gamma:= \begin{cases}\ds\min \left\{ \max\left\{ 1, \frac{1}{\gamma}\right\},  \left(1+   \left\| \frac{1}{a}\right\|_{L^1(0,1)}\right)\right\}, & \gamma \neq 0,\\
	\ds	1+   \left\| \frac{1}{a}\right\|_{L^1(0,1)}, &\gamma =0,
		\end{cases}
		\]
		and
		$C_{a,K, \lambda, \mu}:= \sqrt{A_\gamma}\left(|\lambda|\sqrt{C_1} +|\mu|\right)$. Observe that if $\gamma \ge 1$, then $A_\gamma =1$.

Moreover, observe that if $y$ is a mild solution of \eqref{(P_feed)}, then, by \eqref{star}, \eqref{stimawd}, the fact that $y_x(t,0)=0$ and the definition of $|||\cdot|||$, we also have that
\[
y^2(t,1) \le  \left\| \frac{1}{a}\right\|_{L^1(0,1)} \|\sqrt{a} y''\|_{L^2(0,1)}^2\le 2  \left\| \frac{1}{a}\right\|_{L^1(0,1)}E_y(t),
\]
\[
y_x^2(t,1)  \le  \left\| \frac{1}{a}\right\|_{L^1(0,1)} \|\sqrt{a} y''\|_{L^2(0,1)}^2\le 2  \left\| \frac{1}{a}\right\|_{L^1(0,1)}E_y(t),
\]
\[
y^2(t,1) \le \left\| \frac{1}{a}\right\|_{L^1(0,1)} \|\sqrt{a} y''\|_{L^2(0,1)}^2\le  \left\| \frac{1}{a}\right\|_{L^1(0,1)} |||y|||^2
\]
and
\[
y_x^2(t,1) \le   \left\| \frac{1}{a}\right\|_{L^1(0,1)} \|\sqrt{a} y''\|_{L^2(0,1)}^2\le  \left\| \frac{1}{a}\right\|_{L^1(0,1)} |||y|||^2
\]
for all $t \ge0$. Clearly, these estimates hold also if $\beta$ and $\gamma$ are strictly positive provided $a$ (WD). Thus
\begin{equation}\label{stimamigliore1}
y^2(t,1) \le C_\beta E_y(t),
\end{equation}
\begin{equation}\label{stimamigliore2}
y_x^2(t,1) \le C_\gamma E_y(t),
\end{equation}
\begin{equation}\label{stimamigliore3}
y^2(t,1) \le \frac{1}{2}C_{\beta} |||y|||^2
\end{equation}
and
\begin{equation}\label{stimamigliore4}
y_x^2(t,1) \le \frac{1}{2}C_{\gamma} |||y|||^2,
\end{equation}
where
\[
C_\beta := \begin{cases} \ds2\min\left \{   \left\| \frac{1}{a}\right\|_{L^1(0,1)}, \frac{1}{\beta} \right\}, & \beta \neq 0,\\
 2  \left\| \frac{1}{a}\right\|_{L^1(0,1)}, &\beta=0\end{cases} \text{ and } C_\gamma:=  \begin{cases}\ds2\min\left\{   \left\| \frac{1}{a}\right\|_{L^1(0,1)},  \frac{1}{\gamma}\right\}, & \gamma \neq 0,\\
 2  \left\| \frac{1}{a}\right\|_{L^1(0,1)}, &\gamma=0. \end{cases}
\]
Clearly \eqref{eroeferreonew},  \eqref{stimamigliore1},\eqref{stimamigliore2}, \eqref{stimamigliore3} and \eqref{stimamigliore4} improve \eqref{eroeferreo} and \eqref{stimapun}, \eqref{stimapun0}, \eqref{stimaz1}, \eqref{stimaz2}, respectively, in the case $\beta, \gamma >0$.

Evidently, under the assumptions of this subsection Theorem \ref{teorema energia decr} and Propositions \ref{Prop 3.1}, \ref{Prop 4.4} still hold. On the other hand, Propositions \ref{Prop 3.4} and \ref{Prop 3.3} rewrite as
\begin{Proposition}\label{Prop 3.4new}
		If $y$ is a classical solution of (\ref{(P_feed)}), then there exists $\varepsilon_0>0$ such that for any $0<s<T$ 
			\[
			\begin{aligned}
				\frac{\varepsilon_0}{2}\int_{Q_s}\Biggl (y^2_t+ay^2_{xx} \Biggr )dx\,dt&\le\left( \vartheta + \varrho+\varsigma+ \Biggl (2-\frac{K}{2}\Biggr )\frac{C_\gamma}{2}\right) E_y(s)\\
	&+ \left(\frac{K}{4}+\frac{K\beta}{2}+ \beta\right) \int_s^Ty^2(t,1)dt+ \left(\beta+ 1+ \frac{2\gamma^2}{a(1)}\right)\int_s^Ty_x^2(t,1)dt,
			\end{aligned}
		\]
		where 
		\[\ds\vartheta:=   K\max\Biggl \{1,\frac{2}{a(1)(2-K)}, C_\gamma\Biggr \}, \; \ds\varrho:= 4 \max\left\{ C_\gamma, \frac{2}{a(1)(2-K)}, 1\right\}\]
		and 
		$\ds\varsigma$ is the one defined in \eqref{varsigma}.
	\end{Proposition}

\begin{Proposition}\label{Prop 3.3new}
	If $y$ is a classical solution of (\ref{(P_feed)}), then for every $0<s<T$ and for every  $\delta \in (0, \nu)$
	 we have
		\[
		\begin{aligned}
			\int_s^T y^2(t,1)dt+\int_s^Ty_x^2(t,1)dt&\le\frac{2 \delta}{C_\delta}\int_s^TE_y(t)dt\\
			&+\frac{1}{C_\delta}\left( 2+2C_1^2 A_\gamma^2C_\beta + 2C_1A_\gamma^2 C_\gamma+ \frac{1}{\delta} + 2\frac{C_1A_\gamma^2\max\{1,C_1\}}{\delta}\right)E_y(s),
		\end{aligned}
	\]
	where
	\[
	\nu:= \frac{1}{\max\{ 1,C_1\} (C_{\beta}+ C_{\gamma})A_\gamma},
	\]
$C_1$ is as in \eqref{C1}
	and
		\[
	C_\delta:= 1- \delta\max\{1, C_1\} (C_{\beta}+ C_{\gamma})A_\gamma.
	\]
\end{Proposition}

We omit the proofs since they are similar to the previous ones and, thanks to the previous results, Theorem \ref{teoremaprincipale} still holds with different constants. In particular:
\begin{Theorem}\label{teoremaprincipalenew}
	Assume $a$ (WD) and let $y$ be a mild solution of (\ref{(P)}). Then, for all $t>0$ and for all $
	\ds	\delta \in \left(0,\min \left\{\nu, \frac{\varepsilon_0}{C_3}\right\}\right),$ one has
	\begin{equation}\label{Stabilità}
		E_y(t)\le E_y(0)e^{1-\frac{t}{M}},
	\end{equation}
where
$
M:=\ds\frac{C_4}{\varepsilon_0 - \delta C_3},
$ 

	\[C_3:=\Biggl (\frac{K}{4}+\frac{K\beta}{2}+\beta+\varepsilon_0\frac{\beta}{2}\Biggr )\frac{2}{C_\delta} + \left(\beta +1+\frac{2}{a(1)}\gamma^2+\varepsilon_0\frac{\gamma}{2}\right)\frac{2}{C_\delta}, \]

\[
	C_4:=\vartheta + \varrho+\varsigma+ \Biggl (2-\frac{K}{2}\Biggr )\frac{C_\gamma}{2} + C_5
	\]
	and
\[\begin{aligned}
C_5:=&\ds\frac{1}{C_\delta}\left( 2+2C_1^2 A_\gamma^2C_\beta + 2C_1A_\gamma^2 C_\gamma+ \frac{1}{\delta} + 2\frac{C_1A_\gamma^2\max\{1,C_1\}}{\delta}\right) \cdot\\
&\cdot\left(\frac{K}{4}+\frac{K\beta}{2}+2\beta+\varepsilon_0\frac{\beta}{2} +1+\frac{2\gamma^2}{a(1)}+\varepsilon_0\frac{\gamma}{2}\right).
\end{aligned}\]
Here  $\vartheta, \varrho$, $\varsigma$ and $\nu, C_\delta, C_1$ are considered in Propositions \ref{Prop 3.4new} and \ref{Prop 3.3new}, respectively.
\end{Theorem}

	\section{Appendix}\label{appendice}
	\begin{proof}[Proof of Theorem \ref{th generazione}]
		According to \cite[Corollary 3.20]{nagel}, it is sufficient to prove that $ \mathcal A:D(\mathcal A)\to \mathcal H_0$ is dissipative and that $\mathcal I-\mathcal A$ is surjective, where
		\[
		\mathcal {I}:=\begin{pmatrix}
			Id & 0 \\
			0 & Id
		\end{pmatrix}.
		\]

		\underline{$ \mathcal A$ is dissipative:} take $(u,v) \in D(\mathcal A)$. Clearly, we can apply \eqref{GF0} obtaining
		\[
		\begin{aligned}
			\langle \mathcal A (u,v), (u,v) \rangle_{\mathcal H_0} &=\langle (v, -Au), (u,v) \rangle _{\mathcal H_0}\\
			&=\int_0^1 au''v''dx- \int_0^1 vAu\,dx =0;
		\end{aligned}
		\] 
		thus $\mathcal A$ is dissipative (see
		\cite[Chapter 2.3]{nagel}).
		
		\underline{$\mathcal I - \mathcal A$ is surjective:} 
		take  $(f,g) \in \mathcal H_0$. We have to prove that there exists $(u,v) \in D(\mathcal A)$ such that
		\begin{equation}\label{4.3'}
			( \mathcal I-\mathcal A)\begin{pmatrix} u\\
				v\end{pmatrix} = \begin{pmatrix}f\\
				g \end{pmatrix} \Longleftrightarrow  \begin{cases} v= u -f,\\
				Au + u= f+ g.\end{cases}
		\end{equation}
		To this purpose, introduce the bilinear form $L:\Ho \times \Ho\to \R$ given by
		\[
		L(u,z):=  \int_0^1 u z\, dx + \int_0^1au''z''dx 
		\]
		for all $u, z \in \Ho$.  Thanks  to the equivalence of the norms, the operator $L$ is coercive and continuous: in fact, for all $u,z \in \Ho$, we have
		\[
		|L(u,z)| \le  \|u\|_{ L^2(0,1) }\|z\|_ {L^2(0,1) } +\|\sqrt{a}u''\|_{L^2(0,1)}\|\sqrt{a}z''\|_ {L^2(0,1)}\le 2 \|u\|^2_2\|z\|^2_2.
		\]
			Now, define $F: \Ho \rightarrow \R$ as
		\[
		F(z)=\int_0^1(f+g) z\,  dx,
		\]
		for all $z \in \Ho $.
		Obviously, $F\in \left(\Ho \right)^*$, being $\left(\Ho \right)^*$ the dual space of $\Ho $ with respect to the pivot space $L^2(0,1)$. Hence, by the Lax-Milgram Theorem, there exists a unique solution $u \in \Ho $ of
		\[
		L(u,z)= F(z)  \mbox{ for all }z\in \Ho ,\]
	i.e.
		\begin{equation}\label{4.4'}
			\int_0^1 u z\, dx + \int_0^1 au''z''dx = \int_0^1(f+g) z\, dx 
		\end{equation}
		for all $z \in \Ho $.
		
		Now, take $v:= u-f$; then $v \in \Ho $.
		We will prove that $(u,v) \in D(\mathcal A)$ and solves \eqref{4.3'}. To begin with, \eqref{4.4'} holds for every $z \in \mathcal C_c^\infty(0,1).$ Thus we have
		\[
		\int_0^1 au''z''dx = \int_0^1(f+g-u) z\, dx 
		\]
		for every $z \in\mathcal  C_c^\infty(0,1).$ Hence $\ds (au'')''= (f+g- u)$ a.e. in $(0,1)$, i.e. $
		Au = f+g-u
		$ a.e. in $(0,1)$, and $Au \in L^2(0,1)$. Thus $au'' \in H^2(0,1)$ (by \cite[Lemma 2.1]{CF}), $u \in \mathcal Z(0,1)$ and $u'(0)=u'(1)=0$, if $a$ is (WD), or $u'(1)=0$ if $a$ is (SD) (recall that $u \in H^2_{a,0}(0,1)$). In the last case, coming back to \eqref{4.4'} and  applying the formula of integration by parts \eqref{GG}, we obtain
		\[
\int_0^1 (au'')''z\,dx - (au''z')(0)=  \int_0^1 au''v''dx= \int_0^1(f+g-u) z\, dx 
\]
for all  $z \in \Ho $. Since $\ds (au'')''= (f+g- u)$ a.e. in $(0,1)$, one has $ (au''z')(0)=0$ for all $z \in \Ho $; hence $(au'')(0)=0$. This implies that $(u,v) \in D(\mathcal A)$, $ u +A u =f+g$, and, recalling that $v = u - f$,  we have that $(u,v)$ solves \eqref{4.3'}.
	\end{proof}
	
	\begin{proof}[Proof of Theorem \ref{th generazionenew}]
		Since the proof is similar to the previous one, we sketch it.

	\underline{$ \mathcal B$ is dissipative:} take $(u,v) \in D(\mathcal B)$. Then $(u,v) \in D(B)\times K^2_{a,0}(0,1)$ and so \eqref{GF0new} holds. Hence, 
	\[
	\begin{aligned}
		\langle \mathcal B (u,v), (u,v) \rangle_{\mathcal K_0} &=\langle (v, -Bu), (u,v) \rangle _{\mathcal K_0}\\
		&=\int_0^1 au''v''dx- \int_0^1 vAudx+\beta u(1)v(1)+\gamma u'(1)v'(1)\\
		&=- [(au'')'v](1) +  [au''v'](1)+\beta u(1)v(1)+\gamma u'(1)v'(1)\\
		&=v(1)[\beta u(1)-(au'')'(1)]+v'(1)[\gamma u'(1)+(au'')(1)]=-v^2(1)-(v'(1))^2 \le 0,
	\end{aligned}
	\]
thus  $\mathcal B$ is dissipative.
	
	\underline{$\mathcal I - \mathcal B$ is surjective:} 
	take  $(f,g) \in \mathcal K_0=K^2_{a, 0}(0,1)\times L^2(0,1)$. We have to prove that there exists $(u,v) \in D(\mathcal B)$ such that
	\begin{equation}\label{B}
		( \mathcal I-\mathcal B)\begin{pmatrix} u\\
			v\end{pmatrix} = \begin{pmatrix}f\\
			g \end{pmatrix} \Longleftrightarrow  \begin{cases} v= u -f,\\
			Au + u= f+ g.\end{cases}
	\end{equation}
	Thus, define $G: K^2_{a,0}(0,1) \rightarrow \R$ as
	\[
	G(z)=\int_0^1(f+g) z\,dx+z(1)f(1)+z'(1)f'(1),
	\]
	for all $z \in K^2_{a,0}(0,1)$.
	Obviously, $G\in \left(K^{2}_{a,0}(0,1)\right)^*$, being $\left(K^{2}_{a,0}(0,1)\right)^*$ the dual space of $K^2_{a,0}(0,1)$ with respect to the pivot space $L^2(0,1)$. Now, introduce the bilinear form $\mathcal{P}:K^{2}_{a,0}(0,1)\times K^{2}_{a,0}(0,1)\to \R$ given by
	\[
	\mathcal{P}(u,z):= \int_0^1au''z''dx +\int_0^1 u z \, dx+(\beta +1)u(1)z(1)+(\gamma +1)u'(1)z'(1)
	\]
	for all $u, z \in K^{2}_{a,0}(0,1)$. Clearly, since $\beta$ and $\gamma$ are non negative constants, thanks to the equivalence of the norms given in Proposition \ref{normeequivalenti}, $\mathcal{P}(u,z)$ is coercive. Furthermore, $\mathcal{P}(u,z)$ is continuous: indeed,  by \eqref{star} and \eqref{stimau'(1)}, we have
	\[
	\begin{aligned}
	|\mathcal{P}(u,z)| &\le  \|\sqrt{a}u''\|_{L^2(0,1)}\|\sqrt{a}z''\|_ {L^2(0,1)}+\|u\|_{ L^2 (0,1) }\|z\|_ {L^2 (0,1) }\\
	&+(\beta +1)|u(1)||z(1)|+(\gamma +1)|u'(1)||z'(1)|\\
	&\le \|\sqrt{a}u''\|_{L^2(0,1)}\|\sqrt{a}z''\|_ {L^2(0,1)}+\|u\|_{ L^2 (0,1) }\|z\|_ {L^2 (0,1) }+(\beta +1)\|u'\|_{L^2(0,1)}\|z'\|_{L^2(0,1)}\\
	&+ 4(\gamma+1)\left(\frac{\|\sqrt{a}u''\|^2_{L^2(0,1)}}{a(1)(2-K)}+  \|u'\|^2_{L^2(0,1)}\right)\left(\frac{\|\sqrt{a}z''\|^2_{L^2(0,1)}}{a(1)(2-K)}+ \|z'\|^2_{L^2(0,1)}\right)\\
	&\le C \|u\|_{2,a}\|z\|_{2,a}
	\end{aligned}
	\]
	for all $u,z \in K^{2}_{a,0}(0,1)$ and for a positive constant $C$.
	
	As a consequence, by the Lax-Milgram Theorem, there exists a unique solution $u \in K^{2}_{a,0}(0,1)$ of
	\[
	\mathcal{P}(u,z)= G(z),\]
	for all $z\in K^{2}_{a,0}(0,1)$, namely
	\begin{equation}\label{B1}
	\begin{aligned}
			&\int_0^1 au''z''dx +\int_0^1 u z\,dx+(\beta +1)u(1)z(1)+(\gamma +1)u'(1)z'(1)=\\ &\int_0^1(f+g) z\,dx +z(1)f(1)+z'(1)f'(1)
	\end{aligned}
	\end{equation}
	for all $z \in K^{2}_{a,0}(0,1)$.
	
	Now, take $v:= u-f$; then $v \in K^{2}_{a,0}(0,1)$.
	We will prove that $(u,v) \in D(\mathcal B)$ and solves \eqref{B}. As in the proof of Theorem \ref{th generazione}, one can prove that $u \in \mathcal W(0,1)$.
	
	Coming back to (\ref{B1}) and applying the formula of integration by parts (\ref{GGnew}), we obtain
	
\begin{equation}\label{raikou}
	\begin{aligned}
		&	\int_0^1 (au'')'' z\,dx-(au'')'(1)z(1)+[au''z']_{x=0}^{x=1} +(\beta +1)u(1)z(1)+(\gamma +1)u'(1)z'(1)=\\ &\int_0^1(f+g-u) z\,dx +z(1)f(1)+z'(1)f'(1).
	\end{aligned}
	\end{equation}
	Since $(au'')''=(f+g-u)$ a.e. in $(0,1)$, (\ref{raikou}) implies that
	\begin{equation}\label{B2}
		-(au'')'(1)z(1)+(\beta +1)u(1)z(1)=f(1)z(1)\,\,\,\,\,\quad\forall\;z \in K^{2}_{a}(0,1),
	\end{equation}
	\begin{equation}\label{B3}
	(au'')(1)z'(1)+(\gamma +1)u'(1)z'(1)=f'(1)z'(1)\,\,\,\,\,\quad\forall\;z \in K^{2}_{a}(0,1)
\end{equation}
and
\begin{equation}\label{B4}
(au''z')(0) =0\,\,\,\,\,\quad\forall\;z \in K^{2}_{a}(0,1).
\end{equation}
	From \eqref{B2} and \eqref{B3} it follows
	\begin{equation*}
		-(au'')'(1)+(\beta +1)u(1)=f(1)\Longleftrightarrow 	-(au'')'(1)+\beta u(1)=f(1)-u(1)=-v(1)
	\end{equation*}
	and
	\begin{equation*}
	(au'')(1)+(\gamma +1)u'(1)=f'(1)\Longleftrightarrow 	(au'')(1)+\gamma u'(1)=f'(1)-u'(1)=-v'(1)
\end{equation*}
	respectively, i.e. $-(au'')'(1)+\beta u(1)+v(1)=0$ and $(au'')(1)+\gamma u'(1)+v'(1)=0$.
	On the other hand, if $a$ is (WD), then  $u'(0)=0$ and \eqref{B4} is clearly satisfied; if $a$ is (SD), then \eqref{B4} implies $(au'')(0)=0$.
	Hence $(u,v)$ belongs to $D(\mathcal B)$ and  solves \eqref{B}.
\end{proof}

\end{document}